\documentclass[11pt]{amsart}

\usepackage[USenglish]{babel}

\usepackage{amsmath,amsthm,amssymb,amscd}
\usepackage{booktabs}
\usepackage{url}

\usepackage{MnSymbol}

\usepackage{tikz}
\usetikzlibrary{arrows,shapes.geometric,positioning,decorations.markings}

\usepackage[mathscr]{eucal}
\usepackage[normalem]{ulem}
\usepackage{latexsym,youngtab}
\usepackage{multirow}
\usepackage{epsfig}
\usepackage{parskip}

\usepackage[margin=2cm]{geometry}
\usepackage{color}

\newtheoremstyle{custom}
  {3pt}
  {3pt}
  {\slshape}
  {}
  {\bfseries}
  {.}
  { }
   {}
\theoremstyle{custom}
\newtheorem{theorem}{Theorem}[section]
\newtheorem{proposition}[theorem]{Proposition}

\newtheorem{proposition/definition}[theorem]{Proposition/Definition}
\newtheorem{lemma}[theorem]{Lemma}
\newtheorem{corollary}[theorem]{Corollary}
\newtheorem{conjecture}[theorem]{Conjecture}

\theoremstyle{definition}

\newtheorem{example}[theorem]{Example}

\theoremstyle{remark}
\newtheorem{remark}[theorem]{Remark}

\newtheoremstyle{exercise}
  {3pt}
  {6pt}
  {}
  {}
  {\bfseries}
  {:}
  { }
   {}
\theoremstyle{exercise}
\newtheorem{exercise}[theorem]{Exercise}
\newtheoremstyle{exercises}
  {3pt}
  {6pt}
  {}
  {}
  {\bfseries}
  {:}
  {\newline}
   {}

\theoremstyle{exercise}
\newtheorem{exercises}[theorem]{Exercises}

\input epsf
\def\boxit#1{\vbox{\hrule height1pt\hbox{\vrule width1pt\kern3pt
  \vbox{\kern3pt#1\kern3pt}\kern3pt\vrule width1pt}\hrule height1pt}}

\def\trank{\text{rank}}

\def\BC{\mathbb C}

\def\BP{\mathbb P}
\def\pp#1{\mathbb P^{#1}}

\def\pp#1{{\mathbb P}^{#1}}
\def\tdim{{\rm dim}}

\def\hd{,...,}
\def\ww{\wedge}

\def\inv{{}^{-1}}

\def\cB{{\mathcal B}}\def\cA{{\mathcal A}}

\def\cO{{\mathcal O}}

\def\BB{{\mathbb B}}

\def\ZZ{\mathbb Z}

\def\11{\mathbf 1}

\def\a{\alpha}

\def\o{\omega}

\def\b{\beta}

\def\s{\sigma}
\def\k{\kappa}

\def\d{\delta}

\def\ot{{\mathord{ \otimes } }}
\def\op{{\mathord{\,\oplus }\,}}

\def\ra{{\mathord{\;\rightarrow\;}}}

\def\La#1{\Lambda^{#1}}

\def\op{\oplus}
\def\BZ{\Bbb Z}

\def\ep{\epsilon}
\def\op{\oplus}

\def\ul{\underline}

\def\s{\sigma}

\def\a{\alpha}
\def\b{\beta}

\def\FS{\mathfrak  S}

\def\ol{\overline}

\def\BP{\mathbb  P}
\def\BC{\mathbb  C}

\def\pp#1{\mathbb  P^{#1}}
\def\cC{\mathcal  C}

\def\ep{\epsilon}

\def\hd{, \hdots ,}

\def\inv{{}^{-1}}
\def\ii{|II|}

\def\kk#1#2{k_{{#1}{#2}}}
 
\def\La#1{\Lambda^{#1}}

\def\pp#1{\mathbb  P^{#1}}

\def\ur{\underline{\mathbf{R}}}
\def\uR{\underline{\mathbf{R}}}

\def\ra{\rightarrow}

\def\tdet{\operatorname{det}}

\def\tperm{\operatorname{perm}}

\def\tdim{\operatorname{dim}}

\def\tlim{\lim}

\def\tmin{\operatorname{min}}

\def\trank{\operatorname{rank}}

\def\ww{\wedge}

\def\be{\begin{equation}}
\def\ene{\end{equation}}

\def\ccc{{\mathbf{c}}}
\def\tsgn{{\rm{sgn}}}

\def\trank{\mathbf{R}}

\def\G{\Gamma}

\def\Mone{M_{\langle 1\rangle}}\def\Mtwo{M_{\langle 2 \rangle}}

\def\Mone{M_{\langle 1\rangle}}

\def\Mtwo{M_{\langle 2\rangle}}

\def\trank{{\mathrm {rank}}}

\def\ccc{\mathbf{c}}

\def\nnn{\mathbf{n}}

\def\Mone{M_{\langle
1\rangle}}

\def\Mtwo{M_{\langle 2\rangle}}

\def\trank{{\mathrm {rank}}}

 \def\ccc{{\bold c}}

\def\nnn{\bold n}

\def\BB{\mathbb{B}}

\def\jj{{j'}}\def\kk{{k'}}\def\ii{{i'}}
\def\hj{{\hat j}}\def\hk{{\hat k}}\def\hi{{\hat i}}\def\hjj{{\hat j'}}\def\hkk{{\hat k'}}\def\hii{{\hat i'}}

\keywords{Matrix multiplication complexity, Tensor rank, Asymptotic rank, Laser method}


\renewcommand{\a}{\alpha}
\renewcommand{\b}{\beta}

\renewcommand{\BC}{\mathbb{C}}

\renewcommand{\k}{\kappa}
\renewcommand{\d}{\delta}

 \newcommand{\uQ}{\underline{\mathbf{Q}}}
\newcommand{\aQ}{\uwave{\mathbf{Q}}}

\newcommand{\aR}{\uwave{\mathbf{R}}}
\newcommand{\bfR}{\mathbf{R}}

\renewcommand{\bar}[1]{\overline{#1}}
\renewcommand{\hat}[1]{\widehat{#1}}

\renewcommand{\emptyset}{\font\cmsy = cmsy11 at 11pt
 \hbox{\cmsy \char 59}
}

\newcommand{\textsum}{{\textstyle\sum}}

\begin{document}

\author[A. Conner, H. Huang, J.M. Landsberg]{Austin Conner, Hang Huang, and  J. M. Landsberg}

\address{Department of Mathematics, Texas A\&M University, College Station, TX 77843-3368, USA}
\email[A. Conner]{connerad@math.tamu.edu}
\email[H. Huang]{hhuang@math.tamu.edu}
 \email[J.M. Landsberg]{jml@math.tamu.edu}

\title[Algebraic geometry and the laser method]{Bad and good news for Strassen's laser method:
Border rank of $\tperm_3$  and strict submultiplicativity}

\thanks{Landsberg   supported by NSF grant  AF-1814254}



\begin{abstract}  We   determine  the border ranks of tensors that could 
potentially advance the known upper bound for
the exponent $\o$  of matrix multiplication. The Kronecker
square of the small $q=2$ Coppersmith-Winograd tensor  equals
the $3\times 3$ permanent,   and could potentially be used to show $\o=2$.
We prove the negative result for complexity theory that its  border rank is $16$, resolving a longstanding problem.
Regarding its $q=4$ skew cousin in $\BC^5\ot \BC^5\ot \BC^5$,
which could potentially be used to prove $\o\leq 2.11$,  we show the border rank of its Kronecker
square   is at most $42$,   a remarkable sub-multiplicativity result, as the square
of its border rank is $64$. 
We also determine moduli
spaces  $\ul{VSP}$   for the small Coppersmith-Winograd tensors.
 \end{abstract}

\maketitle

\section{Introduction}
This paper advances both upper and  lower bound techniques in the study of the complexity of tensors
and applies these advances to tensors that may be used to upper bound  the exponent $\o$  of matrix multiplication.

The exponent $\omega$ of matrix multiplication is defined as
\[
\omega :=\inf \{\tau \mid  \text{ two } \nnn \times \nnn \text{ matrices may be multiplied using } O(\nnn^\tau) \text{ arithmetic operations}\}.
 \]
It is a fundamental constant governing the complexity of the basic operations in linear algebra. 
It is generally conjectured that $\omega = 2$. It has been known since 1988 that
$\o\leq 2.38$ \cite{MR91i:68058} which was slightly improved upon 2011-2014  
\cite{stothers,Williams,LeGall:2014:PTF:2608628.2608664}.
All  new  upper bounds on  $\o$ since 1987 have been obtained 
using Strassen's laser method,  which    bounds $\o$  via auxiliary    tensors,
  see any of \cite{MR91i:68058,blaserbook,MR3729273} for a discussion. 
The    bounds of $2.38$ and below  were obtained using  the    big   Coppersmith-Winograd tensor
as  the  auxiliary tensor.
In \cite{MR3388238} it was shown the big Coppersmith-Winograd tensor could not
be used to prove $\o<2.3$ in the usual laser method. 

In this paper we examine $14$ tensors that potentially could be used to prove $\o<2.3$ with the laser method.
Our approach is via algebraic geometry and representation theory, building on the recent advances
in \cite{BBapolar,CHLapolar}.
We solve the longstanding problem (e.g.,    \cite[Problem 9.8]{blaserbook}, 
\cite[Rem. 15.44]{BCS})    
of determining the border rank of the Kronecker square of the 
only Coppersmith-Winograd tensor that could potentially prove $\o=2$
(the $q=2$ small Coppersmith-Winograd tensor). The answer is a negative
result for the advance of upper bounds, as it is $16$, the maximum possible value.
On the positive side, we show that  a tensor that could potentially be
used to prove $\o< 2.11$  has border rank of its Kronecker
square   significantly smaller than the square of its border rank. While this result alone does
not give a new upper bound on the exponent, it opens a promising new direction for upper bounds.
We also develop   new lower   and upper  bound techniques,
and present directions  for future research.

The tensors we study are  the small Coppersmith-Winograd tensor \cite{MR664715}  $T_{cw,q}$ for $2\leq q\leq 10$ (nine such) and
its skew cousin \cite{CGLVkron}  $T_{skewcw,q}$  for even $q\leq 10$ (five such).
(These tensors are defined for larger $q$ but they are only useful for the laser method
when $q\leq 10$.)
The tensors $T_{cw,2}$ and $T_{skewcw,2}$   potentially could be used to prove $\o=2$.
Explicitly,
the small Coppersmith-Winograd tensors  \cite{MR91i:68058} are
$$T_{cw,q}=\sum_{j=1}^q a_0\ot b_j\ot c_j+a_j\ot b_0\ot c_j+a_j\ot b_j\ot c_0
$$
and, for $q=2p$ even,  its skew cousins \cite{CGLVkron} are
$$
T_{skewcw,q}=\sum_{\xi=1}^p a_0\ot  b_\xi\ot c_{\xi+p} - a_0\ot b_{\xi+p}\ot c_\xi 
- a_{\xi}\ot b_0\ot c_{\xi+p} + a_{\xi+p}\ot b_0\ot c_{\xi }
+a_{\xi}\ot b_{\xi+p}\ot c_{0}- a_{\xi+p}\ot b_{\xi }\ot c_{0}.
$$
The small Coppersmith-Winograd tensors are symmetric tensors and their
skew cousins are skew-symmetric tensors. When $q=2$, after a change of basis
$T_{cw,2}$ is just a monomial written as a tensor,
$
T_{cw,2}=\sum_{\s\in \FS_3} a_{\s(1)}\ot b_{\s(2)}\ot c_{\s(3)}
$
and 
$
T_{skewcw,2}=\sum_{\s\in \FS_3}\tsgn(\s) a_{\s(1)}\ot b_{\s(2)}\ot c_{\s(3)}
$. Here $\FS_3$ denotes the permutation group on three elements.

We need the following definitions to state our results:

A tensor $T\in  A\ot B\ot C=\BC^m\ot \BC^m\ot \BC^m$ has {\it rank one} if   
$T=a\ot b\ot c$ for some $a\in A$, $b\in B$, $c\in C$, and
the   {\it rank} of  $T$, denoted $\bold R(T)$,  is the smallest $r$ 
such that $T$ may be written as a sum of $r$ rank one tensors. 
The {\it border rank} of $T$, denoted $\ur(T)$,   is the smallest $r$ such that 
$T$ may be written as a limit of   rank $r$  tensors. In geometric
language, the border rank is  smallest $r$ such that $[T]\in \sigma_r(Seg(\BP 
A\times \BP B\times \BP C))$, where  $\sigma_r(Seg(\BP 
A\times \BP B\times \BP C))$  denotes  the $r$-th secant variety of the Segre variety of 
rank one tensors. 

For symmetric tensors $T\in S^3A\subset A\ot A\ot A$ we may also consider the {\it Waring} or {\it symmetric rank}
of $T$, $\bold R_S(T)$, the smallest $r$ such that $T=\sum_{s=1}^r v_s\ot v_s\ot v_s$
for some $v_s\in A$,  and the Waring
border rank $\ur_S(T)$, the smallest $r$ such that 
$T$ may be written as a limit of  Waring  rank $r$  symmetric tensors.

For tensors $T\in A\ot B\ot C$ and $T'\in A'\ot B'\ot C'$, the {\it Kronecker product} of $T$ and $T'$ is the tensor $T\boxtimes T' := T \ot T' \in (A\ot A')\ot (B\ot B')\ot (C\ot C')$, regarded as $3$-way tensor. Given $T \in A \otimes B \otimes C$, the {\it Kronecker powers}  of $T$ are $T^{\boxtimes N} \in A^{\otimes N} \otimes B^{\otimes N} \otimes C^{\otimes N}$, defined iteratively. Rank and border rank are submultiplicative under Kronecker product: $\bfR(T \boxtimes T') \leq \bfR(T) \bfR(T')$, $\uR(T \boxtimes T') \leq \uR(T) \uR(T')$, and both inequalities may be strict.  

Strassen's {\it  laser method} \cite{MR882307,MR664715} obtains upper bounds on $\o$ by
showing a random degeneration 
of a large Kronecker power of a simple tensor
degenerates further  to a sum of disjoint matrix multiplication
tensors, and then applying  Sch\"onhage's asymptotic sum inequality \cite{MR623057}. 
The relevant results for this paper   are:
 
  For all $k$ and $q$,  \cite{MR91i:68058}
\begin{equation}\label{cwbndq}
\omega\leq  \log_q(\frac 4{27}(\uR(T_{cw,q}^{\boxtimes k}))^{\frac 3k}) .
\end{equation}
For all $k$ and even $q$,  \cite{CGLVkron}  
\begin{equation}\label{skewcwbndq}
\omega\leq  \log_q(\frac{4}{27}(\uR(T_{skewcw,q}^{\boxtimes k}))^{\frac{3}{k}}) .
\end{equation}

Coppersmith-Winograd  \cite{MR91i:68058} showed $\ur(T_{cw,q})=q+2$. Applied to \eqref{cwbndq}
with $k=1$ and $q=8$ gives $\o\leq 2.41$, which was the previous record before $2.38$.

The most natural way to upper bound  the exponent of matrix multiplication would
be to upper bound the border rank of the matrix multiplication tensor directly.
There are very few results in this direction: work of  Strassen \cite{Strassen493}, Bini  \cite{MR592760}, Pan (see, e.g., \cite{MR765701}),  and Smirnov (see, e.g., \cite{MR3146566})
are all we are aware of.
  In order to lower
the exponent further with the matrix multiplication tensor the first
opportunity to do so would be to show  the border rank of the $6\times 6$ matrix multiplication tensor
equaled its known lower bound of $69$ from \cite{MR3842382}.

\subsection{Main Results}
After the barriers of  \cite{MR3388238}, the 
auxiliary tensor  viewed as most promising for upper bounding  the exponent,
or even proving it is two,  is the small Coppersmith-Winograd tensor, or
more precisely its Kronecker powers. In \cite{CGLVkron} bad news in this
direction was shown for the square of most of these tensors and even the cube.
Left open was the square of $T_{cw,2}$ as it was unaccessible by the technology
available at the time (Koszul flattenings and the border substitution method), 
although it was shown that $15\leq \ur(T_{cw,2}^{\boxtimes 2})\leq 16$.
With the advent of border apolarity \cite{BBapolar,CHLapolar}
and the Flag Condition (Proposition \ref{flagbapolar}) that
strengthens it,  we are able to resolve this last open case.
See Remark \ref{perm3prev} for an explanation why this result was previously unaccessible, 
even with the techniques of \cite{BBapolar,CHLapolar}.
The result for the exponent is    negative:

\begin{theorem} \label{perm316} $\ur(T_{cw,2}^{\boxtimes 2})=16$.
\end{theorem} 

In \cite{CGLVkron} it was observed that  $ T_{cw,2}^{\boxtimes 2} =\tperm_3$, the $3\times 3$ permanent considered
as a tensor. Previously   Y. Shitov  \cite{shitovperm3} showed that
the Waring rank of $\tperm_3$ is at least $16$, which matches the upper bound
of \cite{MR3492642}.

Theorem \ref{perm316} is proved in \S\ref{perm316pf}.

We determine the border rank of  $T_{skewcw,q}$ in the range relevant for the laser
method:
  
\begin{theorem} \label{32qupper}
 $\ur(T_{skewcw,q})\leq \frac 32q+2$ and equality holds for $q\leq 10$.
 \end{theorem}
 
 While this is less promising than the equality $\ur(T_{cw,q})=q+2$,
   in \cite{CGLVkron} a significant drop in the border rank
 of $T_{skewcw,2}^{\boxtimes 2}=\tdet_3$ was shown, namely that it
 is $17$ rather than $25=\ur(T_{skewcw,2})^2$.  (The upper bound was shown in \cite{CGLVkron}
 and the lower bound in \cite{CHLapolar}.)
  Theorem \ref{32qupper} implies   $\ur(T_{skewcw,4} )^2=64$. The following theorem
is the largest drop in border rank under a Kronecker square that we are aware
of: 

\begin{theorem}$(*)$\label{skewcw4sect} 
  $\ur(T_{skewcw,4}^{\boxtimes 2})\leq \ur_S(T_{skewcw,4}^{\boxtimes 2}) \leq 42 $.
\end{theorem}

The Theorem is marked with a $(*)$ because the result is only shown to hold
numerically. The expression we give has largest error
  $4.4 \times  10^{-15}$. We could have presented  a solution to 
higher
accuracy, but we were unable to find an algebraic expression.
The new numerical techniques used to obtain this decomposition are described in
\S\ref{skewcw4sect}.
We also give a much simpler Waring  border rank $17$ expression for $\tdet_3=T_{skewcw,2}^{\boxtimes 2}$ than the one in \cite{CGLVkron}, see \S\ref{det3bis}. 

Using  Koszul flattenings (see \S\ref{kflowerbnds}) we show
 $\ur(T_{skewcw,4}^{\boxtimes 2})\geq 39$. For the cube we show
$\ur(T_{skewcw,4}^{\boxtimes 3})\geq 219$ whereas for its cousin
we have $180\leq \ur(T_{cw,4}^{\boxtimes 3})\leq 216$. 
We also prove, using Koszul flattenings,  
lower bounds for $\ur(T_{skewcw,q}^{\boxtimes 2})$    and $\ur(T_{skewcw,q}^{\boxtimes 3})$  for $q\leq 10$. 
(These results are all part of Theorem \ref{kyfbndsprop}.)

\begin{remark}
Starting with the fourth Kronecker power it is possible the border rank of
$T_{skewcw,q}^{\boxtimes 4}$ is less than that of $T_{cw,q}^{\boxtimes 4}$,
  for
    $q\in \{2, 6,8\}$.
The best possible upper bound on $\o$ obtained from some
$T_{skewcw,q}^{\boxtimes 4}$  would be $\omega\leq 2.39001322$ which could potentially be attained
with $q=6$. Starting with the fifth Kronecker power it is potentially possible to beat the current world
record for $\o$ with $T_{skewcw,q}$ and for $T_{cw,q}$ it is already possible
with the fourth power. \end{remark}

Strassen's asymptotic rank conjecture \cite{MR1341854} posits  that for all concise
tensors $T\in \BC^m\ot \BC^m\ot \BC^m$ with regular 
positive dimensional symmetry group  (called {\it tight tensors}),
$\tlim_{k\ra \infty} [\ur(T^{\boxtimes k})]^{\frac 1k}=m$. As a first step towards
this conjecture it is an important problem to determine which tensors $T$ satisfy
$\ur(T^{\boxtimes 2})<\ur(T)^2$. 
We discuss what we understand about this problem in \S\ref{lasergood}.

  A variety that parametrizes all possible
border rank decompositions of a given tensor $T$, denoted $\ul{VSP}(T)$, 
is defined in   \cite{BBapolar}. This variety
naturally sits in a product of Grassmannians, see \S\ref{vspsect} for the definition. We observe that in many examples
     $\ul{VSP}(T)$ often has a large dimension when
$\ur(T^{\boxtimes 2})<\ur(T)^2$ (although not always), and in all examples
we know of,  when $\ul{VSP}(T)$ is zero-dimensional  one also  has 
$\ur(T^{\boxtimes 2})=\ur(T)^2$. This is reflected in the following results:

\begin{theorem}  For $q>2$,  $\ul{VSP}(T_{cw,q})$ is 
a single  point.
\end{theorem}

\begin{theorem} \label{vspcw2} $\ul{VSP}(T_{cw,2})$ 
  consists of three points.
\end{theorem}

More precise versions of these results are given in \S\ref{vspcwq}.

 In contrast $\ul{VSP}(T_{skewcw,q})$ is positive dimensional, at least
 for all $q$ relevant for complexity theory ($q\leq 10$). 
 Explicitly,  $\ul{VSP}(T_{skewcw,2})$ is at least $8$-dimensional,
 see Corollary \ref{r5c3}, and for   $4\leq q\leq 10$, $\tdim\ul{VSP}(T_{skewcw,q})\geq \binom{q/2}2$, see Corollary \ref{vspskewcwq}.

 Border apolarity is just in its infancy. In 
 \S\ref{bapchallengess} we discuss 
 challenges  to getting better results
 with the method   and take first steps  to overcome them
 in  \S\ref{viabilitysect}.  In particular, Proposition  \ref{flagbapolar} was critical to the proof 
 of  Theorem \ref{perm316} as it enables one to substantially reduce the border
 apolarity search space in certain situations (weights occurring with multiplicities).

\subsection{Previous border rank bounds    on $T_{cw,q}^{\boxtimes k} $ and $T_{skewcw,q}^{\boxtimes k}$ } 
\begin{itemize}
\item 
$\ur(T_{cw,q}^{\boxtimes 2})=(q+2)^2$ for $q>2$ and $15\leq \ur(T_{cw,2}^{\boxtimes 2})\leq 16$.
 \cite{CGLVkron}
 
 \item 
$\ur(T_{cw,q}^{\boxtimes 3})=(q+2)^3$ for $q>4$.  
 \cite{CGLVkron}
 
 \item $\ur(T_{skewcw,2}^{\boxtimes 2})=17$. \cite{CHLapolar}
 
 \item $\ur(T_{skewcw,q})\geq q+3$.  \cite{CGLVkron}
 \item For all $q>4$ and all $k$,  $\ur(T_{cw,q}^{\boxtimes k})\geq (q+2)^3(q+1)^{k-3}$  and $\ur(T_{cw,4}^{\boxtimes k})\geq 36(5)^{k-2}$. \cite{CGLVkron}
 
 \item $\ur(T_{cw,2}^{\boxtimes k})\geq 15(3^k)$.  \cite{CGLVkron}
  
 \end{itemize}

With the exception of the proof $\ur(T_{skewcw,2}^{\boxtimes 2})\geq 17$,  which was
obtained via border apolarity, 
 these  lower bounds were obtained using Koszul flattenings.  
 
 Previous to  these
 it was shown that  $\ur(T_{cw,q}^{\boxtimes k})\geq (q+1)^k+2^k-1$
 using the border substitution method  \cite{MR3578455}.

\subsection{Definitions/Notation}\label{defs}
Throughout, $A,B,C$ will denote complex vector spaces   of dimension $m$.
We let $\{ a_i\}$ denote a basis of $A$, with either $0\leq i\leq m-1$ or
$1\leq i\leq m$ and similarly for $\{ b_j\}$ and $\{ c_k\}$. 
The dual space to $A$ is denoted $A^*$.
The $\BZ$-graded algebra of symmetric tensors  is denoted $Sym(A)=\oplus_d S^dA$,  it is also the
algebra of homogeneous polynomials on $A^*$.
  For $X\subset A$, $X^\perp:=\{\a\in A^*\mid 
\a(x)=0\forall x\in X\}$ is its annihilator, and  $\langle X\rangle\subset A$ denotes the span of $X$.  
Projective space is  $\BP A= (A\backslash \{0\})/\BC^*$, and if $x\in A\backslash \{0\}$, we let $[x]\in \BP A$ denote
the associated point in projective space (the line through $x$).
The general linear group of invertible linear maps $A\ra A$ is denoted $GL(A)$
and the special linear group of determinant one linear maps is denoted $SL(A)$. 
The permutation group on $r$ elements is denoted $\FS_r$. 

The Grassmannian of $r$ planes through
the origin is denoted $G(r,A)$, which we will view in its Pl\"ucker embedding $G(r,A)\subset \BP \La r A$.

For a set $Z\subset \BP A$, $\ol{Z}\subset \BP A$ denotes its Zariski closure,
$\hat Z\subset A$ denotes the cone over $Z$ union the origin, $I(Z)=I(\hat Z)\subset Sym(A^*)$ denotes
the ideal of $Z$,
and $\BC[\hat Z]=Sym(A^*)/I(Z)$, denotes the homogeneous coordinate ring of $\hat Z$.
Both  $I(Z)$, $\BC[\hat Z]$ are $\BZ$-graded by degree.

We will be dealing with ideals on products of three  projective spaces, that is
we will be dealing with polynomials that are homogeneous in three sets of variables,
so our ideals with be $\BZ^{\op 3}$-graded. More precisely, we will study 
ideals $I\subset Sym(A^*)\ot Sym(B^*)\ot Sym(C^*)$, and $I_{stu}$ denotes
the component in $S^sA^*\ot S^tB^*\ot S^uC^*$.

For $T\in A\ot B\ot C$, define the {\it symmetry group} of $T$,
$G_T:=\{g=(g_1,g_2,g_3)\in GL(A)\times GL(B)\times GL(C)
\mid g\cdot T=T\}$.

Given $T,T' \in A \otimes B \otimes C$, we say that $T$ \emph{degenerates} to $T'$ if $T' \in \bar{GL(A) \times GL(B) \times GL(C) \cdot T}$, the closure of the orbit of $T$, equivalently in the Euclidean or   Zariski topology.

Given $T\in A\ot B\ot C$, we may consider it as a linear map $T_C: C^*\ra A\ot B$, and we let $T(C^*)\subset A\ot B$
denote its image, and similarly for permuted statements. A tensor $T$ is {\it $A$-concise} if
the map  $T_A $ is injective, i.e., if it requires all basis vectors in  $A$ to
write down in any basis, and $T$ is {\it concise} if it is $A$, $B$, and $C$ concise. 
A tensor is {\it $1_A$-generic} if $T(A^*)\subset B\ot C$ contains an element of
maximal rank $m$.


\subsection{Acknowledgements} We thank J. Buczy\'{n}ski and J. Jelisiejew  for very 
helpful conversations and J. Buczy\'{n}ski for extensive  comments on a draft
of this article.
 
 \section{Border apolarity and the challenges it faces}
\subsection{Border apolarity}\label{bapolarsect}
Given $T\in A\ot B\ot C$ and $r$, 
{\it border apolarity} \cite{BBapolar} gives a sequence of necessary conditions for
$\ur(T)\leq r$: one successively defines a $\BZ^{\op 3}$-graded ideal
$I\subset Sym(A^*\op B^*\op C^*)$
such that each $I_{stu}\subset S^sA^*\ot S^tB^*\ot S^uC^*$ has codimension $r$ for all $s+t+u>1$.
One first requires $I_{110}\subseteq T(C^*)^\perp$, $I_{101}\subseteq T(B^*)^\perp$, 
$I_{ 011 }\subseteq T(A^*)^\perp$ and $I_{ 111 }\subset T^\perp$. 
Because $I$ is an ideal, it must contain the images of the multiplication maps $I_{s-1,t,u}\ot 
A^*\ra S^sA^*\ot S^tB^*\ot S^uC^*$,  $I_{s ,t-1,u}\ot 
B^*\ra S^sA^*\ot S^tB^*\ot S^uC^*$, 
 $I_{s ,t,u-1}\ot 
C^*\ra S^sA^*\ot S^tB^*\ot S^uC^*$,  which places conditions on the ranks of the sum  of 
the three maps,
which is called the $(stu)$-map and the rank condition the $(stu)$-test.
Write $E_{stu}=I_{stu}^\perp$. It will
be convenient to phrase the codimension tests dually:
\begin{proposition}\cite[Prop. 3.1]{CHLapolar}
The $(210)$-test is passed if and only if   
  skew-symmetrization map 
\be\label{e210map}
E_{110}\ot A \ra \La 2 A\ot B 
\ene
has kernel of dimension at least $r$. The kernel is $(E_{110}\ot A)\cap (S^2A\ot B)$. 

The $(stu)$-test is passed if and only if 
the  triple intersection
\be\label{110inter} (E_{s,t,u-1} \ot C)\cap (E_{s,t-1,u}\ot B)\cap(E_{s-1,t,u}\ot A) 
\ene
has  dimension at least $r$. 
\end{proposition}


We will make repeated use of the following lemma:

\begin{lemma}[Fixed ideal Lemma  \cite{BBapolar}] \label{fil}
If $T$ has symmetry group $G_T$, then the ideal above may be required to be fixed under the action
of a Borel subgroup of $G_T$ which we will denote $\BB_T$. In particular, if $G_T$ contains a torus, it may be 
required to be fixed
under the action of the torus.
\end{lemma}

Border apolarity provides both lower bounds and a guide to proving upper bounds.
For example, the $(111)$ space for $T_{skewcw,q}$ described in
the proof of  Theorem \ref{32qupper} hints at the formula \eqref{brskewcw}, where
the terms linear in $t$ appear in the $(111)$ space.

\subsection{Challenges facing border apolarity}\label{bapchallengess}
In modern algebraic geometry the study of geometric objects (algebraic varieties)
is replaced by the study of the  ideal of  polynomials that vanish on a variety.
The study of a set of $r$ points $\{z_1\hd z_r\}$
 in affine space $\BC^N$ is replaced by the study of
its ideal, more precisely the quotient 
$\BC[x_1\hd x_N]/I_{z_1\sqcup \cdots \sqcup z_r}$ where $\BC[x_1\hd x_N]$ is the ideal of all polynomials
on $\BC^N$ and $I_{z_1\sqcup \cdots \sqcup z_r}$ is the ideal.
Note that ring  $\BC[x_1\hd x_N]/I_{z_1\sqcup \cdots \sqcup z_r}$
is a vector space of dimension $r$, called the coordinate ring of the variety.
(In our
case we will be concerned with $r$ points on the Segre variety
$Seg(\BP A\times \BP B \times \BP C)$ but the issues about
to be discussed are local and there is no danger working in affine space.)
The study becomes one of such rings, and one no longer requires them
to correspond to ideals of points, only that the vector space has dimension
$r$ and that the ideal is saturated. Such are called {\it zero dimensional schemes
of length $r$}. If the ideal corresponds to $r$ distinct points one says the scheme
is {\it smooth}.   A central challenge of border apolarity as a tool
in the study of border rank, is that applied na\"\i vely, it only determines
necessary conditions for an ideal to be  the limit of a sequence of such ideals.
  One could split the problem
of detecting non-border rank ideals into two: first, just get
rid of the   ideals that are not limits of ideals
of zero dimensional schemes,  then, given an ideal that is a limit of
  ideals of zero dimensional schemes of length $r$, determine if
  it is a limit of ideals of smooth schemes ({\it smoothability} conditions). In this paper we address the first problem
and the new additional necessary conditions we obtain (Proposition \ref{flagbapolar}) are enough
to enable us to determine   $\ur(T_{cw,2}^{\boxtimes 2})$ via border apolarity.
In \S\ref{badnews} we show that     ideals that fail to deform to saturated
ideals  occur already for quite low border rank.
 The second problem is ongoing work with J. Buczy\'{n}ski and his group in Warsaw.  

The second problem is a serious issue:
The {\it cactus   rank}  \cite{BBapolar,MR3121848}  of a tensor $T$   
is the smallest $r$ such that $T$ lies in the span of a zero dimensional scheme
of length $r$ supported on the Segre variety. The cactus border rank of $T$, $\ul{\bold{CR}}(T)$ is
the smallest $r$ such that $T$ is a limit of tensors of cactus rank $r$.
One has $\ur(T)\geq \ul{\bold{CR}}(T)$ and for almost all  tensors the inequality is strict.
The $(stu)$ tests are   tests for   cactus border rank.
Cactus border rank is not known to be relevant for complexity theory, thus the failure
of current border apolarity technology to distinguish between them is a barrier to future progress.
Moreover, the cactus variety fills the ambient space at latest
border rank  $6m-4$, see 
\cite[Ex. 6.2 case $k=3$]{gazka2016multigraded}.

\subsection{Viability and the flag conditions}\label{viabilitysect}

We begin in the general context of secant varieties with a preliminary observation:

For a projective variety $X\subset \pp N$, define its variety of
secant $\pp{r-1}$'s,
$$
\s_r(X):=\ol{ \bigcup_{x_1\hd x_r\in X} \langle x_1\hd x_r\rangle }.
$$

\begin{proposition} Let $X\subset \BP V$ be a projective
variety and let $\BP E\subset \s_r(X)$ be a $\pp{r-1}$ 
arising from a border rank $r$ decomposition of a point on $\s_r(X)$.
Then there exists a complete flag $E_1\subset E_2\subset \cdots
\subset E_r=E$ such that for all $1\leq j\leq r$,  $\BP E_j\subset \s_j(X)$.
\end{proposition}
\begin{proof}
We may write
$E=\tlim_{t\ra 0}\langle x_1(t)\hd x_r(t)\rangle$ where $x_j(t)\in X$ and the limit
is taken in the Grassmannian $G(r,V)$ (in particular, for all $t\neq 0$ we
may assume $x_1(t)\hd x_r(t)$ are linearly independent).
Then take $E_j=\tlim_{t\ra 0} \langle x_1(t)\hd x_j(t)\rangle$
 where the limit
is taken in the Grassmannian $G(j,V)$.
\end{proof}

Let  $T\in A\ot B\ot C$ and let   $E_{stu}$ be   an $r$-dimensional space
 that is  $I_{stu}^\perp$ for a   multi-graded ideal that passes all border apolarity
 tests  up to total degree $s+t+u+1$.

  Call such an ideal, or an $E_{stu}$,  {\it viable} if it arises from 
an actual border rank decomposition.
This implies  $\BP E_{stu}\subset
\s_r( Seg(v_s(\BP A)\times v_t(\BP  B)\times v_u( \BP  C)))$.
Here  $v_s: \BP A\ra \BP (S^sA)$ is the Veronese re-embedding,
$v_s([a])=[a^s]$.

To a $\ccc$-dimensional subspace    $E\subset A\ot B$, one may associate
a tensor $T\in A\ot B\ot \BC^\ccc$, well-defined up to isomorphism, such that
$T(\BC^{\ccc *})=E$. Much of the lower bound literature exploits this correspondence to
reduce questions about tensors to questions about linear subspaces of spaces of matrices.
(This idea appears already in \cite{Strassen505}.)
The following proposition exploits this dictionary to obtain new conditions for viability of
candidate $E_{stu}$'s:

\begin{proposition}\label{flagbapolar}[Flag conditions]
 If $E_{110}$ is viable, then there exists a $\BB_T$-fixed filtration
 of $E_{110}$,  $F_1\subset F_2\subset \cdots \subset F_r=E_{110}$,
 such that $F_j\subset \s_j(Seg(\BP A\times \BP B))$.
 Let $T_j\in A\ot B \ot \BC^j$ be a tensor equivalent to the subspace $F_j$.
 Then $\ur(T_j)\leq j$.

 Similarly, if   $E_{stu}$ is viable,  there are complete flags in $A,B,C$ such that $\BP (E_{stu}\cap S^sA_j\ot S^tB_j\ot S^uC_j)
  \subset  \s_j(Seg(v_s(\BP A)\times v_t(\BP   B)\times v_u( \BP  C)))$ for all $1\leq j\leq m$.
 \end{proposition}

 \begin{proof}
Set $\hat C=C\op \BC^{r-m}$. Then there exists 
  $\hat T\in A\ot B\ot \hat C$ such that $\hat T(\hat C)=E_{110}$ and $\ur(\hat T)\leq r$. 
In this case the flag condition  \cite[Cor. 2.3]{MR3682743} implies  that
since  $\hat T\in A\ot B\ot \hat C= \BC^m\ot \BC^m\ot \BC^r$ with $r\geq m$ is concise of
  minimal border rank $r$,  there exists a  complete  flag 
     $C_1\subset C_2\subset \cdots \subset C_r =\hat C^{*}$ such
  that $\hat T(C_k)\subset \s_k(Seg(\BP A\times \BP B))$. Take $F_k=\hat T(C_k)$.
The proof that the flag may be taken to be Borel fixed is the same as in the
Fixed ideal lemma. 

The second assertion follows from the preceding discussion.
  \end{proof}
  
 Proposition \ref{flagbapolar} provides  additional conditions $E_{stu}$ must satisfy for viability  beyond  the border apolarity tests.
It  allows one to utilize the known conditions for minimal border rank in a non-minimal border rank setting.

 When $T_j$ is concise, Proposition \ref{flagbapolar} is quite
 useful as there are many 
  known conditions for concise tensors to be of minimal border rank. In particular it must have symmetry Lie algebra
 of dimension at least $2j-2$ and if it is $1_{\BC^j}$-generic (for any of the factors), it must
 satisfy the End-closed condition (see \cite{MR3682743}).

 \begin{remark} Proposition \ref{flagbapolar} also applies to cactus border rank decompositions,
 so it is a \lq\lq non-deformable to saturated\rq\rq\  removal  condition rather than a smoothability one.
 \end{remark}

By the classification of tensors of border rank at most three \cite[Thm. 1.2(iv)]{MR3239293} the possibilities for the first two filtrands of $E_{110}$ are
$F_1=\langle a\ot b\rangle$,
$F_{2a}=\langle a\ot b, a'\ot b'\rangle$ or $F_{2b}=\langle a\ot b, a\ot b'+a'\ot b\rangle$
corresponding to either two distinct rank one points or a rank one point and a tangent vector,
and there are five possibilities for $F_3$: 
\begin{enumerate}
\item 
$F_{3aa}=\langle a\ot b, a'\ot b', a''\ot b''\rangle$ (three distinct points) 
\item  $F_{3aab}=\langle a\ot b, a\ot b'+a'\ot b,   a''\ot b''\rangle$
(two points plus a tangent vector to one of them) 
\item  $F_{3bc}=\langle a\ot b, a\ot b'+a'\ot b,   a''\ot b +a'\ot b'+ a\ot b''\rangle$ (points of the
form $x(0),x'(0),x''(0)$ for a curve $x(t)\subset Seg(\BP A\times \BP B)$)
\item $F_{3abb}=\langle a\ot b, a\ot b'+a'\ot b, a\ot b''+a''\ot b\rangle$ (point plus two tangent vectors)
\item  $F_{3bd}=\langle a\ot b, a \ot b', a'\ot b +a''\ot b'+a\ot b''\rangle$ (sum of tangent vectors to two colinear
points $x'+y'$) or its mirror
$F_{3bd}=\langle a\ot b, a' \ot b , a \ot b' +a'\ot b''+a''\ot b \rangle$.
\end{enumerate}

 The space $E_{110}$ contains a distinguished subspace $T(C^*)$.
 Write $E_{110}'$ for a choice of a   complement to $T(C^*)$ in $E_{110}$.

\begin{corollary} \label{via110}  If $E_{110}$ is viable and $\BP T(C^*)\cap \s_k(Seg(\BP A\times \BP B))=\emptyset$, then 
there exists a choice of  $E_{110}'$ such that $ F_k \subset  E_{110}'$.
\end{corollary}
\begin{proof} Say otherwise, then there exists  $M\in  F_k \cap T(C^*)$. This contradicts
 $T(C^*)\cap \s_k(Seg(\BP A\times \BP B))=\emptyset$.
 \end{proof}

\begin{corollary}\label{emptycor} If $\BP T(C^*)\cap \s_q(Seg(\BP A\times \BP B))=\emptyset$, then $\ur(T)\geq m+q$.
\end{corollary}

Although we have stronger lower bounds, 
Corollary \ref{emptycor} provides the following \lq\lq for free\rq\rq :

\begin{corollary}\label{cwscwk}
For all $k$,  $\ur(T_{cw,2}^{\boxtimes k})\geq 3^k+2^k$ and $\ur(T_{skewcw,2}^{\boxtimes k})\geq 3^k+2^k$.
\end{corollary}
\begin{proof}  Let $i_\a,j_\b\in \{ 1,2,3\}$. Then
$$T_{cw,2}^{\boxtimes k}(C^*)
=\langle
\sum_{\s\in \BZ_2^k} \s\cdot (a_{i_1\hd i_k}\ot b_{j_1\hd j_k})
\mid
i_{\a}\neq j_{\a} \forall  1\leq \a\leq k\rangle
$$
and the action of $\s$ is by swapping indices.
This transparently is of rank bounded below by $2^k$.
The case of $T_{skewcw,2}^{\boxtimes k}$ is the same except that the coefficients appear with
signs.
\end{proof} 

\subsection{Pure and mixed kernels}\label{pmker}
Define three types of contribution to the kernel of
the $(210)$-map:  the {\it free} kernel
$$
\k_f:=\tdim [(T(C^*)\ot A)\cap (S^2A\ot B)]
$$
the {\it pure } kernel 
$$\k_p=\tmin_{E_{110}'}\tdim[(A\ot E_{110}')\cap (S^2A\ot B)] ,
$$
where the min is over all choices of $E_{110}'\subset E_{110}$,
and the {\it  mixed}  kernel  
$$\k_m=\tdim[(A\ot E_{110})\cap (S^2A\ot B)] -\k_p-\k_f
$$ 
corresponding to
elements of the kernel arising from linear combinations of elements of $A\ot E_{110}'$
and $A\ot T(C^*)$.  In this language,  $E_{110}'$ passes  the $(210)$ test if and only if
   $\k_p+\k_m\geq r- \k_f$.
Define corresponding $\k_p',\k_m'$ for the $(120)$-test.
  
  \begin{conjecture} If $E_{110}$ is such that $\k_m,\k_m'=0$, then it is not viable.
  \end{conjecture}
  
  Intuitively, if $E_{110}'$ never \lq\lq sees\rq\rq\ the tensor, it should not be viable.


\subsection{Limits of the total degree $3$ border apolarity tests}\label{badnews}

\begin{proposition}\label{weakdeg3} Let $m \geq  9$, but $m\neq 10,15$. Then for any tensor
in $\BC^m\ot \BC^m\ot \BC^m$, there are candidate ideals   passing all
degree three tests for border rank at most $r$  when $r\geq 2m$.  

More generally, setting $r=m+k^2$, there are candidate ideals in total degree two passing   all degree three tests  
once $m\leq  \frac{k^3}2-\frac{k^2}2$.
In particular, 
for all $\ep>0$,  $r\geq m+m^{\frac 13 +\ep}$,  and $m$ sufficiently large,  there   are such  candidate ideals.
\end{proposition}
\begin{proof}
For the first assertion, it suffices to prove the case $r=2m$ and   the tensor $T$ is concise.  Set $k=\lfloor \sqrt{m}\rfloor$, $t=k+\lceil \frac{m-k^2}2\rceil$,  and $t'=k+\lfloor \frac{m-k^2}2\rfloor$.
Take
$E_{110}'=\langle a_1\hd a_k\rangle\ot \langle b_1\hd b_k\rangle+\langle a_{k+1}\hd a_{t' }\rangle \ot b_1
+a_1\ot \langle b_{k+1}\hd b_{t}\rangle$
and similarly for the other spaces.
Then 
\begin{align*}
&(E_{110}\ot A )\cap (S^2A \ot B)\supseteq\\
& S^2\langle a_1\hd a_k\rangle\ot \langle b_1\hd b_k\rangle
\op \langle a_{k+1}\hd a_{t'}\rangle\cdot \langle a_1\hd a_k\rangle\ot b_1
\op S^2 \langle a_{k+1}\hd a_{t'}\rangle\ot b_1+a_1^{\ot 2}\ot \langle b_{k+1}\hd b_t\rangle.
\end{align*}
This has dimension $\binom{k+1}2 k+ (t'-k)k+\binom{t'-k+1}2+(t-k)$ which is at least $2m$ in the specified range.
(The only value greater than $8$ the inequality  fails for is $m=10$.)
Similarly the $(120)$ test is passed at least as easily.
Finally
$$
(E_{110}\ot C)\cap (E_{101}\ot B)\cap (E_{011}\ot A)\supseteq
\langle a_1\hd a_k\rangle
\ot
\langle b_1\hd b_k\rangle
\ot 
\langle c_1\hd c_k\rangle
\op \langle T\rangle
$$
which has dimension $k^3+1$ which is at least $2m$ in the range of the proposition.
(The only value   greater than $8$ the inequality  fails for is $m=15$.)

The second assertion follows with the same $E_{110}'$,  taking $r=m+k^2$ and  $t,t'=0$.

\end{proof}
 
\begin{example} For $T_{cw,2}^{\boxtimes 3}$ \label{box3}
   it is easy to get $E_{110}'$ of dimension $21$ (so for border rank $48<63$) that pass the
  $(210)$ and $(120)$ tests. Take $E_{110}'$ spanned by rank one basis vectors such that
  the associated Young diagram is a staircase.
  Then $\k_p=\k_p'=1(6)+2(5)+3(4)+4(3)+5(2)+6(1)= 56>48$. 
\end{example}

\section{Moduli and submultiplicativity}
\subsection{Moduli spaces $\ul{VSP}$}\label{vspsect}

  Following \cite{BBapolar}, define
$\ul{VSP}(T)$ to be the set of ideals as in \S\ref{bapolarsect} arising from a border rank $\ur(T)$ decomposition
of $T$. (In the notation of \cite{BBapolar} this is $\ul{VSP}(T, \ur(T))$.) Since
for zero dimensional  schemes of a fixed length (and more generally for schemes
with a fixed Hilbert polynomial), there is
a uniform bound on degrees of generators of their ideals,
this is a finite dimensional variety which naturally embedds in a product of Grassmannians.  

A more classical object also of interest is $\ul{VSP}_{A\ot B\ot C}(T)\subset G(\ur(T), A\ot B\ot C)$, which
just records the $\ur(T)$-planes giving rise to a border rank decomposition, i.e., the
annihilator of the  $(111)$-component
of the ideal. In particular $\tdim(\ul{VSP}_{A\ot B\ot C}(T))\leq \tdim (\ul{VSP}(T))$.

It will be useful to state the following result in a more general context:
Let $X\subset \BP V$ be a variety not contained in a hyperplane,
assume $\s_{r-1}(X)\neq \BP V$ and write $\tdim \s_r(X)=r\tdim (X)+r-1-\d$.
 Consider the incidence correspondence
 $$
S_r(X)=\ol{\{ ((x_1\hd x_r), y, V)\in X^{\times r}\times 
\BP V\times G(r,V)
\mid
y\in \langle x_1\hd x_r\rangle\subseteq V\}}, 
$$
and its projection maps
 $$
 \begin{matrix}
& & S_r(X) & & \\
& \swarrow &  &\searrow &\\
G(r,V) & &&& \s_r(X).
\end{matrix}
$$
Call the projections $\pi_G,\pi_\s$.
We have $\tdim S_r(X)=r\tdim(X)+r-1$ so for $y\in \s_r(X)_{general}$,
 $\tdim(\pi_\s^{-1}(y))=\d$.
 
 Define $\ul{VSP}_{X,\BP V}(y):=\pi_G\pi_\s^{-1}(y)$.
 When $X=Seg(\BP A\times \BP B\times \BP C)$, $y=T$,  and $V=A\ot B\ot C$,
 this  is $\ul{VSP}_{A\ot B\ot C}(T)$.
 
 \begin{proposition} For all $y\in \s_r(X)$,
  $\tdim \ul{VSP}_{X,\BP V}(y)\geq \delta$.
\end{proposition}

\begin{proof} 
By \cite{MR2427466} $\tdim\pi_G(S_r(X))=r\tdim(X)$, so $\pi_G$ 
generically has $(r-1)$-dimensional fibers,
which correspond to the choice of a point in $\BP V$.
This  implies that $\pi_G|_{{\pi_\s}\inv(y)}$ is finite to one. Since $\tdim \pi_\s^{-1}(y)\geq \d$
we conclude.
\end{proof}

 \begin{corollary}\label{r5c3}  A border rank five tensor   $T\in \BC^3\ot \BC^3\ot \BC^3$
 has $\tdim \ul{VSP}_{A\ot B\ot C}(T)\geq 8$.
 \end{corollary}
 \begin{proof}
 $\tdim \s_5(Seg(\pp 2\times \pp 2\times \pp 2))=26$.
 \end{proof}
 
 \begin{remark} In this case, by \cite{MR2652318} $T$ also has rank five and thus 
 $\tdim VSP_{A\ot B\ot C}(T)\geq 8$, where $VSP_{A\ot B\ot C}(T)$ is the space
 of rank decompositions.
 \end{remark}

 A similar argument shows:

\begin{proposition} \label{gposvsppos} Let $\cO $ be the largest over all
$(s,t,u)$ of the  smallest
orbit of $G_T$ in $\BP (S^sA\ot S^tB\ot S^uC)$.
Then $\tdim \ul{VSP}(T)\geq \tdim \cO $.
\end{proposition}

\subsection{How to find good tensors for the laser method?}\label{lasergood}

The utility of a tensor $T\in A\ot B\ot C$  for the laser method may be thought of as the ratio of its {\it cost}, which is the asymptotic
rank,
$
\aR(T) := \lim_{N \to \infty} [\uR(T^{\boxtimes N})] ^{1/N},     $
  and
its {\it value},  which is   its asymptotic subrank $\aQ(T) := \lim_{N \to \infty} [\uQ(T^{\boxtimes N})]^{1/N}$. 
See \cite{MR3984631} for a discussion, where the ratio of their
logs  is called the {\it irreversibility} of $T$. 
Here  $\uQ(T)$ is the maximum $q$ such that  $\Mone^{\oplus q}
\in \ol{GL(A)\times GL(B)\times GL(C)\cdot T}$.
Unless
a tensor is of minimal border rank, we only can estimate the asymptotic rank  of a tensor by computing
its border rank and the border rank of its small Kronecker powers.

 There are several papers attempting to find tensors that give good upper bounds on $\o$
in the laser method: 

Papers on {\it barriers} may be interpreted as describing where {\it not} to look for good tensors:
\cite{MR3984631,MR3984617,DBLP:conf/innovations/AlmanW18,2018arXiv181008671A}
discuss  limits of the laser method for various types of tensors and various types of implementations.

A program to utilize algebraic geometry and representation theory to find good tensors for the laser method was initiated in
\cite{2019arXiv190909518C,MR3682743}.

Here we describe a more modest goal:
determine criteria that indicate (or even guarantee) that border rank is  
strictly sub-multiplicative under the Kronecker square.

To our knowledge, the first example of a non-minimal border rank tensor that
satisfied $\ur(T^{\boxtimes 2})=\ur(T)^2$ was given in \cite{CGLVkron}: the small Coppersmith-Winograd
tensor $T_{cw,q}$ for $q>2$ and in this paper we show equality also holds when $q=2$. 
This shows that tight tensors need not exhibit strict submultiplicativity.
Several examples of strict submultiplicativity were known previous to this paper: the $2\times 2$ matrix
multiplication tensor $\Mtwo\in \BC^4\ot \BC^4\ot \BC^4$, $ \ur(\Mtwo)=7$ \cite{MR2188132}
while  $\ur(\Mtwo^{\boxtimes 2})\leq 46$ \cite{2014arXiv1412.1687S}.
The tensors   of \cite{MR3941923} have a drop of one, a generic tensor   $T\in \BC^3\ot \BC^3\ot \BC^3$
satisfies $\ur(T)=5$ while  $\ur(T^{\boxtimes 2})\leq 22$ \cite{CGLVkron}, and  $\ur(T_{skewcw,2})=5$ while 
$\ur(T_{skewcw,2}^{\boxtimes 2})=17$ \cite{CGLVkron,CHLapolar}.

All the strict submultiplicativity examples have positive dimensional $\ul{VSP}$. This 
is   attributable to
the degeneracy of $\s_4(Seg(\pp 2\times\pp 2\times \pp 2))$ for the generic tensors in $\BC^3\ot \BC^3\ot \BC^3$,
and to the large symmetry groups for the other cases: If a tensor
$T\in A\ot B\ot C$  has a positive dimensional symmetry group $G_T$
and $G_T$ does not have a one-dimensional submodule in each of $A\ot B$, $A\ot C$, $B\ot C$, $A\ot B\ot C$,
then $\tdim(\ul{VSP}(T))>0$ because any ideal in the $G_T$-orbit closure of an ideal of a border rank decomposition
for $T$ will give another border rank decomposition.

It would be too much to hope that
   a concise tensor $T$ not of minimal border rank satisfying $\tdim \ul{VSP}(T)>0$ 
also satisfies $\ur(T^{\boxtimes 2})<\ur(T)^2$. Consider the following example:
Let $T=T_1\op T_2$ with the $T_j$ in disjoint spaces, where
 $T_1$ has non-minimal border rank and $\tdim \ul{VSP}(T_1)=0$ and $T_2$ has minimal border rank with 
 $\tdim \ul{VSP}(T_2)>0$. Then there is no reason to believe $T^{\boxtimes 2}$
 should have strict submultiplicativity. 
  
It is possible that the converse holds: that
strict  submultiplicativity under the Kronecker square implies a positive dimensional
$\ul{VSP}$.

It might be useful, following \cite{MR3941923} to split the submultiplicativity question into
two questions: first to determine if the usual tensor square is submultiplicative and
then if the border rank of the Kronecker square is less than the border rank of the
tensor square.
Note that in general, assuming non-defectivity, for a projective
variety $X\subset \BP V$ of dimension $N$,  $\s_{R-1}(X)$
has codimension $N+1$ in $\s_R(X)$. In our case $R=r^2$ and in the tensor
square case $N=6m-6$, and in the Kronecker square case $N=3m^2-3$.
A priori,  for $T\in \BC^m\ot \BC^m\ot \BC^m$ of border rank $r$,
  $T^{\otimes 2} \in \s_{r^2}(Seg(\BP^{(m-1)\times 6}))$ and 
  submultiplicativity  is a codimension
  $6m-5$ condition, whereas 
  $T^{\boxtimes 2} \in \s_{r^2}(Seg(\BP^{(m^2-1)\times 3}))$ and 
  submultiplicativity  is a codimension
  $3m^2-2$ condition.  Despite this, the second condition is weaker than the first.

\section{Koszul flattening lower bounds}\label{kflowerbnds}
The best general technique available for border rank lower bounds are Koszul flattenings \cite{MR3376667,MR3081636}.

Fix   an integer $p$. Given a tensor $T = \sum_{ijk} T^{ijk} a_i \otimes b_j \otimes c_k \in A \otimes B \otimes C$, the $p$-th \emph{Koszul flattening} of $T$ on the space $A$ is the linear map
\begin{align*}
T_A^{\ww p}: \Lambda^p A \ot B^* &\to \Lambda^{p+1}A\ot C \\
X\ot \beta & \mapsto \textsum_{ijk}T^{ijk}\beta(b_j)(a_i\wedge X) \ot c_k.
\end{align*}
Then \cite[Proposition 4.1.1]{MR3081636} states 
\begin{equation}\label{kozinq}
\ur(T)\geq \frac{\trank(T_{A}^{\ww p})}{\binom {\tdim(A)-1}p}.
\end{equation}
The best lower bounds for any given $p$ are obtained by restricting $T$ to a generic
$2p+1$ dimensional subspace of $A^*$ so the denominator becomes $\binom{2p}p$.

\begin{theorem}\label{kyfbndsprop} The following border rank lower bounds are obtained by
applying Koszul flattenings to a restriction of the tensor to a sufficiently
generic $\BC^{2p+1}\ot B\ot C\subset A\ot B\ot C$. Values of $p$ that give
the bound are in parentheses.
\begin{enumerate}
 \item $\ur(T_{skewcw,4}^{\boxtimes 2})\geq 39$ ($p=2,3,4$)
 \item $\ur(T_{skewcw,6}^{\boxtimes 2})\geq 70$ ($p=2,3,4$)
  \item $\ur(T_{skewcw,8}^{\boxtimes 2})\geq 110$ ($p=4$)
    \item $\ur(T_{skewcw,10}^{\boxtimes 2})\geq 157$ ($p=4$)
    \item $\ur(T_{skewcw,2}^{\boxtimes 3})\geq 49$ ($p=4$)
 \item $\ur(T_{skewcw,4}^{\boxtimes 3})\geq 219$ ($p=3$)
 \item $\ur(T_{skewcw,6}^{\boxtimes 3})\geq 550$ ($p=3$)
  \item $\ur(T_{skewcw,8}^{\boxtimes 3})\geq 1089$ ($p=3$)
    \item $\ur(T_{skewcw,10}^{\boxtimes 3})\geq 1886$ ($p=3$).
\end{enumerate}
\end{theorem}

Better lower bounds for the larger cases are potentially possible, if not easily accessible,  using larger values of $p$.

Compare these with the values for the small Coppersmith-Winograd tensor  from \cite{CGLVkron}:
\begin{enumerate}
 \item $\ur(T_{cw,4}^{\boxtimes 2})=36$  
 \item $\ur(T_{cw,6}^{\boxtimes 2})=64$  
  \item $\ur(T_{cw,8}^{\boxtimes 2})= 100$  
    \item $\ur(T_{cw,10}^{\boxtimes 2})= 144$  
 \item $\ur(T_{cw,4}^{\boxtimes 3})\geq 180$  
 \item $\ur(T_{cw,6}^{\boxtimes 3})=512$  
  \item $\ur(T_{cw,8}^{\boxtimes 3})= 1000$  
    \item $\ur(T_{cw,10}^{\boxtimes 3})= 1728$.  
\end{enumerate}

Note  that $\ur(T_{cw,q}^{\boxtimes 4})\leq (q+2)^4$ 
and that $\ur(T_{skewcw,q}^{\boxtimes 4})$ is at least the estimate
in Proposition \ref{kyfbndsprop} times $q+1$ by \cite[Prop. 4.2]{CGLVkron}. 
Based on this, it is possible as of this writing that  
of $\ur(T_{skewcw,q}^{\boxtimes 4})\leq \ur(T_{cw,q}^{\boxtimes 4})$ for $q=2, 6,8$.

\section{$\tperm_3=T_{cw,2}^{\boxtimes 2}$} \label{perm316pf}

\begin{proof}[Proof of Theorem \ref{perm316}] The upper bound follows as $\ur(T_{cw,2})=4$.

For the lower bound, we prove there is no   $E_{110}\subset A\ot B$ of dimension $15$ passing the $(210)$
and $(120)$ tests.
We assume there is one and obtain a contradiction.  

In this proof $\{i,j,k\}=\{1,2,3\}$, $\{\ii,\jj,\kk\}=\{1,2,3\}$, $\hkk \in \{\ii,\jj\}$,    $\hk \in \{ i,j\}$ etc..

 Here 
$$\tperm_3(C^*)=\langle a^i_\ii\ot b^\hi_\hii + a^\hi_\ii \ot b^i_\hii
+  a^i_\hii\ot b^\hi_\ii + a^\hi_\hii\ot b^i_\ii \rangle.
$$
Thus $\tperm_3(C^*)\cap \s_3=\emptyset$.  
Observe that $\kappa_f=1$ as
$$(\tperm_3(C^*)\ot A)\cap (S^2A\ot B)=\langle \sum a^{k}_{\kk}\ot (
a^{i}_{\ii}\ot b^{j}_{\jj}+a^{j}_{\ii}\ot b^{i}_{\jj}+a^{i}_{\jj}\ot b^{j}_{\ii}+a^{j}_{\jj}\ot b^{i}_{\ii}) 
\rangle .
$$

\begin{remark} In general, for any symmetric tensor $T$, $\kappa_f\geq 1$ due
to the copy of $T$ in $S^3A\subset S^2A\ot B$.
\end{remark}

We next discuss what weight vectors can appear in $E_{110}'$.

The possible weights of elements in $A\ot B$ are $(200)(200)$, $(110)(110)$, $(200)(110)$ and their permutations
under the action of $(\FS_3\times \FS_3)\rtimes \BZ_2$. We will say an element has
{\it type} $(xyz)(pqr)$ if its weight is in the $(\FS_3\times \FS_3)\rtimes \BZ_2$-orbit of
 $(xyz)(pqr)$. We fix a weight basis of $E_{110}$ (note that if there are no multiplicities, this
 is already determined up to scale). 
 
 All weight vectors of type $(200)(200)$ have rank one,   vectors of type $(200)(110)$ have
 rank one or two and    vectors of type $(110)(110)$ have rank at most four.

 The rank two weight vectors tangent to a rank one element of
 type $(200)(200)$, which we may write as
$a^j_\jj\ot b^j_\jj$,  are $a^j_\jj\ot b^j_\hjj+Ka^j_\hjj\ot b^j_\jj$, for
   some $K\neq 0$,  or its $\BZ_2$-image,
which are of type  $(200)(110)$
 or $a^j_\jj\ot b^\hj_\hjj+Ka^\hj_\hjj\ot b^j_\jj$, which   are of type    $(110)(110)$.
   
  No rank two tangent vector to a rank one element of type $(110)(110)$  is a weight vector.
  
The   rank two tangent weight vectors to a rank one element of type $(200)(110)$, e.g., 
$a^i_\ii\ot b^i_\kk$,  are of the form
$a^i_\ii\ot b^\hi_\kk+K a^\hi_\ii\ot b^i_\kk$ for  some $K\neq 0$, and they are of  
  type $(110)(110)$.

Consider rank three second derivatives of a weight vector:
Say we had some 
$$
(a +ta'+t^2a'')\ot (b +tb'+ t^2b'')=a\ot b+t(a\ot b'+a'\ot b)
+t^2(a \ot b''+ a'\ot b'+a''\ot b )+...
$$
For the $t$ coefficient to appear in $E_{110}'$,
either we would need $a'\ot b+a\ot b'$ to be a weight vector, or
that either  $a\ot b'$ or $a'\ot b$   already appeared alone in $E_{110}'$.
Say we were in the first case,
that would imply that  
$a'$ is a multiple of $a$ and $b'$ is a multiple of $b$ so our projective curve would just be stationary to first order and we are reduced to a first derivative. 
In the second case,   for the coefficient of $t^2$ to be a weight vector, it must be of
type $(110)(110)$, the weight of $a'\ot b'$, which we write as $a^i_\jj\ot  b^j_{\ii}$.
Thus for a cost of four, we obtain a rank $3$ vector of type $(110)(110)$.
If $a\ot b$ is of type $(200)(200)$ we may write the space as 
\be\label{4plane}
\langle a^i_\ii\ot b^i_\ii, a^i_\jj\ot b^i_\ii,  a^i_\ii\ot b^j_\ii,
a^i_\jj\ot b^j_\ii + a^i_\ii\ot b^j_\jj + a^j_\jj \ot b^i_\ii
\rangle.
\ene
The elements are respectively of type $(200)(200), (200)(110),(200)(110), (110)(110)$.
It is also possible to have $a\ot b$ of type $(200)(110)$ or $(110)(110)$ and obtain 
a rank $3$ vector of type $(110)(110)$,  but it will be clear
from what follows that this is not advantageous.

Since no two vectors of type $(200)(200)$ are colinear, the only possibility of a vector
 of the form $x'+y'$ appearing is if $x$ is of type $(200)(200)$ and $y$ of type $(200)(110)$,
 but in this case one just gets a weight vector of the form $x'$ of type $(110)(110)$ by the weight vector requirement.
 If $x,y$ are of type $(200)(110)$, then any nontrivial $x'+y'$ is not a weight vector.

Thus we are reduced to several overlapping  types of vectors:
rank one weight vectors ($81$ such),  rank two tangent vectors to the nine  rank one vectors
of type $(200)(200)$,     rank two tangent
vectors tangent to a rank one vector of type $(200)(110)$, and 
 a rank three vector of type $(110)(110)$ as long
as we include the four dimensional space described above.

Recall the definitions of the pure and mixed kernels from \S\ref{pmker}.
Here we need  to show it is not possible to have $\k_p+\k_m\geq 14$ and $\k_p'+\k_m'\geq 14$

The proof will be completed by  Lemmas \ref{lem1} and \ref{lem2} below.
\end{proof}

\begin{lemma}\label{lem1} For all choices of admissible $E_{110}'$, $\k_m,\k_m'\leq 4$.\end{lemma}

\begin{lemma}\label{lem2}For all choices of admissible $E_{110}'$, 
$\k_p,\k_p'\leq 10$ and  if equality holds, then  $\k_m,\k_m' <4$.
\end{lemma} 

\begin{proof}[Proof of Lemma \ref{lem1}]
Let $\G=\BZ_2\times \BZ_2$
and  $\G\cdot a^j_\jj\ot b^k_\kk= a^j_\jj\ot b^k_\kk+ a^j_\kk\ot b^k_\jj+ a^k_\jj\ot b^j_\kk+ a^k_\kk\ot b^j_\jj$.
In what follows underlined terms are elements of $E_{110}'$.
(The group $G_{\tperm_3}$ allows us to unambiguously define
the elements of $E_{110}'$ except those of type $(110)(110)$.) 

Up to $(\FS_3\times \FS_3)\rtimes \BZ_2$, there are four types of relations a single element of $\tperm_3(C^*)$
can be in a mixed term:

First, $E_{110}'$ could contain a complement to its weight space (the complement is three dimensional),
which would give rise to a single mixed term by taking a linear combination that reduces its rank to one.
Three type $(110)(110)$ terms of the same weight appear.
For example, if the $(110)(110)$ terms are rank one elements, then the element is  
\be\label{m1}
a^j_\jj\ot (\G\cdot a^i_\ii\ot b^j_\jj) 
-a^j_\jj\ot \ul{(a^i_\ii\ot b^j_\jj )}  
-a^j_\jj\ot \ul{(  a^i_\jj\ot b^j_\ii ) }
- a^j_\jj\ot \ul{(  a^j_\ii\ot b^i_\jj )}  .
\ene
 If there is a rank two element appearing, it cannot be a tangent vector,
 and thus its cost is two and the situation is the same as above.

 The one advantageous situation is if we have a four plane of the form \eqref{4plane}.
 Then we may have the relation
 \be\label{m14}
a^j_\jj\ot (\G\cdot a^i_\ii\ot b^j_\jj) 
-a^j_\jj\ot \ul{(a^i_\ii\ot b^j_\jj + a^i_\jj\ot b^j_\ii+   a^j_\ii\ot b^i_\jj )}  
\ene
giving a gain of one.

 The other possibilities are (here and below, $K,K'$ are constants):
 
\be\label{m2}
a^j_\jj\ot (\G\cdot  a^\hj_\hjj\ot b^j_\jj)+a^\hj_\hjj\ot \ul{(a^j_\jj\ot b^j_\jj) }+
a^\hj_\jj\ot \ul{(a^j_\jj\ot b^j_\hjj+K a^j_\hjj\ot b^j_\jj)} + a^j_\hjj\ot \ul{(a^j_\jj\ot b^\hj_\jj+K a^\hj_\jj\ot b^j_\jj ) } ,
\ene
where the underlined terms are 
respectively types $(200)(200)$, $(200)(110)$, $(200)(110)$,

\begin{align}\label{m3}
&a^i_\ii\ot (\G\cdot a^k_\kk\ot b^j_\jj)+
a^k_\kk\ot \ul{(a^i_\ii\ot b^j_\jj+Ka^j_\jj\ot b^i_\ii)} +
a^k_\jj\ot \ul{ (a^i_\ii\ot b^j_\kk+K' a^j_\kk\ot b^i_\ii)} \\
&\nonumber+ a^j_\kk\ot \ul{(a^i_\ii\ot b^k_\jj+K'a^k_\jj\ot b^i_\ii )}
+ a^j_\jj\ot \ul{(a^i_\ii\ot b^k_\kk+Ka^k_\kk\ot b^i_\ii)},
\end{align}
where the underlined terms are 
all of type $(110)(110)$ and they all  live in different weight spaces, and

 \begin{align}\label{m33}
&a^j_\ii\ot (\G\cdot a^\hj_\kk\ot b^j_\jj)+
a^\hj_\kk\ot \ul{(a^j_\ii\ot b^j_\jj+K a^j_\jj \ot b^j_\ii)} +
a^\hj_\jj\ot \ul{ (a^j_\ii\ot b^j_\kk+K' a^j_\kk\ot b^j_\ii)} \\
&\nonumber+ a^j_\kk\ot \ul{(a^j_\ii\ot b^\hj_\jj+K'a^\hj_\jj\ot b^j_\ii )}+ a^j_\jj\ot \ul{(a^j_{\ii}\ot b^\hj_\kk+Ka^\hj_\kk \ot b^j_\ii)},
\end{align}
where the underlined terms are 
respectively of  types $(200)(110)$, $(200)(110)$, $(110)(110)$, $(110)(110)$.
The two elements of type $(110)(110)$ live in different weight spaces. 

The possible terms in  the mixed kernel   with two  elements of $\tperm_3(C^*)$ that are not
just sums of terms with one element having no cancellation are as follows:

\begin{align}\label{m4}
&a^i_i\ot (\G\cdot a^k_k\ot b^i_j) + a^i_k\ot (\G\cdot a^i_i\ot b^k_j)\\
&+\nonumber
a^k_k\ot \ul{(a^i_i\ot b^i_j+Ka^i_j\ot b^i_i)} +
a^k_j\ot \ul{ (a^i_i\ot b^i_k+  a^i_k\ot b^i_i)} + a^k_i\ot \ul{(a^i_k\ot b^i_j)}
+ a^i_j\ot \ul{(a^i_i\ot b^k_k+Ka^k_k\ot b^i_i)}
+ a^i_j\ot \ul{(a^i_k\ot b^k_i)},
\end{align}
where the underlined terms are 
respectively of   types $(200)(110)$, $(200)(110)$, $(200)(110)$, $(110)(110)$, $(110)(110)$
and
the two type $(110)(110)$ vectors live in different weight spaces, and 

\begin{align}\label{m5}
&a^i_i\ot (\G\cdot a^k_k\ot b^i_j)+ a^i_j\ot (\G \cdot a^i_i\ot b^k_k) \\
&+\nonumber
a^k_k\ot \ul{(a^i_i\ot b^i_j+a^i_j\ot b^i_i)} +
a^k_j\ot \ul{ (a^i_i\ot b^i_k+K' a^i_k\ot b^i_i)} + a^i_k\ot \ul{(a^i_i\ot b^k_j+K'a^k_j\ot b^i_i )}
+a^i_k\ot \ul{(a^i_j\ot b^k_i)}+ a^k_i\ot \ul{(a^i_j\ot b^i_k)},
\end{align}

where the underlined terms are 
respectively of   types $(200)(110)$, $(200)(110)$, $(110)(110)$, $(110)(110)$ and
the two type $(110)(110)$ vectors live in different weight spaces.

The possible terms in  the mixed kernel   with three  elements of $\tperm_3(C^*)$ are:
  \begin{align}& \label{m6}
a^i_\ii\ot (\G\cdot a^\hi_\kk\ot b^i_\jj)+ a^i_\kk\ot (\G\cdot a^i_\ii\ot b^\hi_\jj)+a^i_\jj\ot (\G\cdot a^i_\ii\ot b^\hi_\kk)\\
&\nonumber +
a^\hi_\kk\ot \ul{(a^i_\ii\ot b^i_\jj+a^i_\jj\ot b^i_\ii) }+
a^\hi_\jj\ot \ul{(a^i_\ii\ot b^i_\kk+a^i_\kk\ot b^i_\ii)}+  a^\hi_\ii\ot \ul{(a^i_\kk\ot b^i_\jj+a^i_\jj\ot b^i_\kk) },
  \end{align}
where all have type $(200)(110)$,  with the same $(200)$ weight space and the three distinct $(110)$ weight spaces, and

 \begin{align}& \label{m66}
a^i_\ii\ot (\G\cdot a^j_\jj\ot b^k_\kk)+ a^j_\jj\ot (\G\cdot a^i_\ii\ot b^k_\kk)+a^k_\kk\ot (\G\cdot a^i_\ii\ot b^j_\jj)\\
&\nonumber +
a^i_\ii\ot \ul{(a^j_\kk\ot b^k_\jj+a^k_\jj\ot b^j_\kk) }+
a^j_\jj\ot \ul{(a^i_\kk\ot b^k_\ii+a^k_\ii\ot b^i_\ii)}+  a^k_\kk\ot \ul{(a^i_\jj\ot b^j_\ii+a^j_\ii\ot b^i_\jj) },
  \end{align}
where all have type $(110)(110)$ and lie in different weight spaces.

Since $\tdim E_{110}'=6$  and there are no pairs of relations with an overlap of
more than two vectors from $E_{110}'$, we conclude that any relation using
more than $3$ elements of $E_{110}'$ will not be useful. Thus
we consider the relations \eqref{m1},\eqref{m14},\eqref{m2}, \eqref{m6},  and \eqref{m66}.
If a relation of type \eqref{m1} appears, we have three elements of
type $(110)(110)$ all of the same weight in $E_{110}'$. The only
other relations among those we consider that
 use elements of type $(110)(110)$ are of type \eqref{m66},
 but this requires two additional elements of that type with different
 weights. We conclude a relation of type \eqref{m1} will not be useful, nor
 one of type \eqref{m66}.
 If a relation of type \eqref{m14} appears, then we have a four-plane
 of the form \eqref{4plane}. The best choice of such will allow us
 a relation of the form \eqref{m2} at a cost of two more, which fills the six plane.
 If we take the four-plane \eqref{4plane}, then for a cost of two, we may
 have a relation of type \eqref{m2} for both the $(210)$ and $(120)$ tests,
 explicitly, we need to include $a^i_\ii\ot b^i_\hii,a^\hi_\ii\ot b^i_\ii $ to
 get
 $$
a^i_\ii\ot (\G\cdot  a^\hi_\hii\ot b^i_\ii)+a^\hi_\hii\ot \ul{(a^i_\ii\ot b^i_\ii) }+
a^\hi_\ii\ot \ul{(a^i_\ii\ot b^i_\hii+  a^i_\hii\ot b^i_\ii)} + a^i_\hii\ot \ul{(a^i_\ii\ot b^\hi_\ii
+  a^\hi_\ii\ot b^i_\ii ) } 
$$
and its $(120)$ cousin. If we only add one of these we just get one of the two.
At this point we have either determined $E_{110}'$ or have one choice
of element remaining, and   there is no way to obtain two
(let alone $3$) more elements in $\k_m$ with just one more element.
We are reduced to considering relations of the form \eqref{m2} and \eqref{m6}.
The terms in \eqref{m6} are tangent vectors to an element of type $(200)(200)$
and to have all three terms, we would need at least two such elements
and we only have one more element to choose, which would not
be enough to produce even one more  relation. We are reduced to considering
relations of type \eqref{m2}, which need an element of type $(200)(200)$. 
Take one such element   $a^j_\jj\ot b^j_\jj$. 
It is clear we should take mixed kernel terms that it appears in. There are four such
and to obtain them, we only need include the four vectors
$a^j_\jj\ot b^j_\hjj+  a^j_\hjj\ot b^j_\jj$, $a^j_\jj\ot b^\hj_\jj+K a^\hj_\jj\ot b^j_\jj$ (where there
are two possibilities for each of the  the two hatted vectors).
When $K=1$, this gives $\k_m,\k_m'=4$ (if $K\neq 1$, then $\k_m'$ is smaller). 
At this point we have a five-plane in $E_{110}'$, but by this discussion, there is no
way to increase $\k_m$ further by adding just one more  element.
\end{proof}

For the proof of Lemma \ref{lem2} we use a basic fact from exterior differential systems (the easy
part of Cartan's test) \cite[Prop. 4.5.3]{MR3586335}: 
let
$B_1\subset B_2\subset \cdots \subset B_9$ be a generic flag in $B$ (generic in the sense that
$s_1$ below is maximized, and having maximized $s_1$, $s_2$ is maximized etc..). Let
  $s_1$   be the  dimension of the projection of $E_{110}'$ to $A\ot B_1$.  
Define $s_2$ by  $s_1+s_2$ is    dimension of the projection of 
$E_{110}'$ to $A \ot B_2$,
set $s_1+s_2+s_3$ to be the   dimension of $E_{110}'$  projected to
$A \ot B_3$ 
    etc.. Then 
  \be\label{ctest}\tdim (S^2A\ot B)\cap (A\ot E_{110}')
  \leq s_1+2s_2+\cdots +6s_6.
\ene
  In particular, $\k_p\leq s_1+2s_2+\cdots +6s_6$.
  If equality holds in \eqref{ctest}, we will say  $E_{110}'$ is $A$-involutive.
     
  \begin{proof}[Proof of Lemma \ref{lem2}]
  Let $t_1\hd t_6$ be the corresponding quantities for the $(120)$ test.
  The only way to have $\k_m,\k_m'\geq 10$ is  if the associated Young diagrams
  are  respectively  $A$ and $B$ involutive,  and both  the staircase, i.e.,  
  $s_1=t_1=3,s_2=t_2=2,s_3=t_3=1$. This is because involutivity can only hold if $E_{110}'$ is spanned
  by rank one elements and in this case the $B$-diagram is the transpose of the $A$-diagram.

  Thus there are $a_1,a_2,a_3\in A$ and $b_1,b_2,b_3\in B$ such that
  $$E_{110}'=\langle a_1\ot b_1, a_1\ot b_2, a_2\ot b_1, a_1\ot b_3,a_3\ot b_1,a_2\ot b_2\rangle.
  $$
  The best one can do here is to obtain $\k_m=1$, e.g., by  taking 
  $a_1=a^1_1$,   $a_2=a^1_3$, $a_3=a^3_1$ and similarly for the $b_j$.
  \end{proof}

 \begin{remark}\label{perm3prev} The reason $\tperm_3$ was previously unaccessible
 was that already to choose $E_{110}'$, without the flag condition one needed to
 introduce numerous parameters due to the high weight multiplicities that made the calculation infeasible.
 The flag condition guaranteed the presence of low rank elements in $E_{110}'$
 which significantly reduced the search space.
 \end{remark}

 \begin{remark} It is interesting to see what happens when $\tdim E_{110}'=7$,  to obtain a  border rank $16$
  ideal fixed by the torus in $G_{\tperm_3}$.
  One may take for example
  $$
  E_{110}'= 
  \langle a^1_1\ot b^1_1, a^1_2\ot b^1_1+ a^1_1\ot b^1_2,
  a^1_3\ot b^1_1+ a^1_1\ot b^1_3,
  a^2_1\ot b^1_1+a^1_1\ot b^2_1, a^3_1\ot b^1_1+a^1_1\ot b^3_1, a^1_2\ot b^1_2, a^1_3\ot b^1_2+
  a^1_2\ot b^1_3\rangle.
  $$
  Then we obtain the four $(200)(200)$ contributions to $\k_m$ from
  expressions of type  \eqref{m2} as well as  three additional
  contributions from expressions of type \eqref{m6}. Here $s_1=t_1=4$, $s_2=t_2=3$ and
 \begin{align*}
  &(A\ot E_{110}')\cap (S^2A\ot B)=\\
  \langle &
  a^1_1\ot a^1_1\ot b^1_1, a^1_2\ot a^1_1\ot b^1_1+ a^1_1\ot (a^1_2\ot b^1_1+ a^1_1\ot b^1_2),
  a^1_3\ot a^1_1\ot b^1_1+a^1_1\ot ( a^1_3\ot b^1_1+ a^1_1\ot b^1_3),\\
  &
  a^2_1\ot a^1_1\ot b^1_1+a^1_1\ot (a^2_1\ot b^1_1+a^1_1\ot b^2_1), a^3_1\ot a^1_1\ot b^1_1+a^1_1\ot (a^3_1\ot b^1_1+a^1_1\ot b^3_1),  \\
  &
  a^1_2\ot a^1_2\ot b^1_2, a^1_1\ot a^1_2 \ot b^1_2+ a^1_2\ot (a^1_2\ot b^1_1+ a^1_1\ot b^1_2),
  a^1_3\ot a^1_2\ot b^1_2+ a^1_2\ot (a^1_3\ot b^1_2+
  a^1_2\ot b^1_3)
  \rangle
  \end{align*}
  so $\k_p=\k_p'=8$ and both the $(210)$ and $(120)$ tests are passed.
  \end{remark}

\section{$\ul{VSP}(T_{cw,q})$}\label{vspcwq}
In this section we adopt the index range $1\leq \a,\b\leq q$. 
The small Coppersmith-Winograd tensor  has a well-known border rank decomposition, which is also a Waring border rank decomposition.
\begin{align*}
&T_{ cw,q} =\tlim_{t\ra 0}\\
&
\frac 1{t^2}
\sum_{{\a}} \left[ (a_0+ta_{{\a}})\ot (b_0+tb_{{\a}})\ot (c_0+tc_{{\a}})
 \right] \\
&-\frac 1{t^3}
\left[
(a_0+t^2\sum_{\a} a_{{\a}})\ot (b_0+t^2\sum_{\a} b_{{\a}})\ot (c_0+t^2\sum_{\a} c_{{\a}})\right]\\
&
-( q \frac 1{t^2}-\frac 1{t^3})a_0\ot b_0\ot c_0.
\end{align*}

Let $q>2$.
Write $A=B=C=L\op M$, where
$L=\langle a_0\rangle$ and $M=\langle a_{\a}\rangle$. Set $Q=\sum_{\a} a_{\a}\ot a_{\a}$. We have the group $G=G_{T_{cw,q}}= SO(M,Q)\times GL(L)\rtimes \FS_3=SO(q)\times \BC^*\rtimes \FS_3$.
Then
$$
A\ot B=L^{\ot 2} \op   L \ww M  \op S^2_0M\op \La 2 M \op  (L\cdot M \op Q),
$$
where the term in parenthesis is $T_{cw,q}(C^*)$. 
In what follows we write $L^k$ for $L^{\ot k}=S^kL$.

\begin{theorem}\label{vspcwq} For $q>2$,  $\ul{VSP}(T_{cw,q})$
is a
   point.  The  unique ideal
is as follows:  for all $s,t,u$ with $s+t+u=d$, the annhilator of the ideal in degree $(s,t,u)$ is
$$
L^{  d}\op  L^{d-1}\cdot M\op  L^{d-2}\cdot Q.
$$
Here 
$$L^{d-1}\cdot M=\langle a_0^{s-1}\cdot  a_{\a} \ot b_0^t \ot  c_0^u +
  a_0^{s }\ot b_0^{t-1}\cdot  b_{\a} \ot  c_0^u +  a_0^{s }\ot b_0^{t } \ot c_0^{u-1}\cdot  c_{\a} \mid {\a}=1\hd q\rangle
  $$
and 
$$
L^{d-2}\cdot Q =\langle\sum_{\a} a_0^{s-1}\cdot  a_{\a} \ot b_0^{t-1}\cdot b_{\a}\ot c_0^u+
a_0^{s-1}\cdot  a_{\a} \ot b_0^{t }\ot c_0^{u-1}\cdot c_{\a}+ 
  a_0^{s }\ot b_0^{t-1}\cdot  b_{\a} \ot c_0^{u-1}\cdot  c_{\a}  \rangle.
  $$

\end{theorem}

\begin{proof} We must have $\BP E_{110}\cap Seg(\BP A\times \BP B)\neq \emptyset$.
This may be achieved by adding
some 
$$
(u_{0}a_0+\sum_{\a} u_{\a} a_{\a})\ot(v_0 b_0+ \sum_\b v_\b b_\b)
$$
for $u_0,u_{\a},v_0,v_\b\in \BC$.
 We also must have a flag as in Observation \ref{flagbapolar}. Taking anything
other than $a_0\ot b_0$ or $xa_0\ot b_{\a}+ya_{\a}\ot b_0$ (i.e., some $a_0\ot b_{\a}$ or $a_{\a}\ot b_0$ since
we are working modulo $T(C^*)$)  
makes the flag condition $ \BP F_2 \subset \s_2(Seg(\BP A\times \BP B))$ fail.
(Here we use that $q>2$.) 
Taking anything
other than $a_0\ot b_0$ 
makes the flag condition $ \BP F_3 \subset \s_3(Seg(\BP A\times \BP B))$ fail.
Thus there is a unique $E_{110}$, and by symmetry unique $E_{101}$ and $E_{011}$.
This triple exactly passes all degree three tests.

To see $E_{200}$ must be as asserted, it must be such that $(E_{200}\ot B)\supseteq E_{210}$.
In order to have $L^{\ot 3}$ in this intersection, we need $L^{\ot 2}\subset E_{200}$.
In order to have $ L^2\cdot M=\langle a_0\ot a_0\ot b_{\a}+ a_0\ot a_{\a}\ot b_0+a_{\a}\ot a_0\ot b_0\rangle$
in the intersection,
we see it must also contain $\langle  a_0\ot a_{\a} +a_{\a}\ot a_0 \rangle= L\cdot M$.
In order to have $ L\cdot Q=\langle\sum_{\a} (a_0\ot  a_{\a}\ot b_{\a} + a_{\a}\ot   a_0\ot b_{\a} +a_{\a}\ot a_{\a}\ot b_0)\rangle$
in the intersection,
we see it must also contain $\langle  \sum_{\a}   a_{\a}  \ot   a_{\a}  \rangle=  Q$.

For the general case, assume by induction $E_{s-1,t,u},E_{s,t-1,u},E_{s,t,u-1}$ are as asserted 
and isomorphic as a module to 
$L^{\ot d-1}\op  L^{d-2}\cdot M\op  L^{d-3}\cdot Q$. 
Arguing as we did for $E_{200}$, first obtaining $L^{\ot d}$, then $ L^{d-1}\cdot M$, then 
$ L^{d-2}\cdot Q$ we conclude.
\end{proof}

Note that the   ideal is $G_{T_{cw,q}}$-fixed  as indeed  it has to be if $\ul{VSP}$ is a point.

Now let $q=2$, in this case it is more convenient to write $T_{cw,2}$ as
$$
T_{cw,2}=\sum_{\s\in \FS_3} a_{\s(1)}\ot b_{\s(2)}\ot c_{\s(3)}.
$$
Write $A=B=C=L_1\op L_2\op L_3$ where, e.g., for $A$, $L_j=\langle a_j\rangle$.
We have $\hat G_{T_{cw,2}}=(\BC^*)^{\times 3}\rtimes \FS_3$.

\begin{theorem} \label{vspcw2} $\ul{VSP}(T_{cw,2})$ and $\ul{VSP}_{v_3(\pp 2), \BP S^3\BC^3}(T_{cw,2})$ 
each consists of three points.
One choice has for all $s,t,u$ with $s+t+u=d$, annihilator in degree $(s,t,u)$ equal to
$$
L_1^{s}\ot L_1^t\ot L_1^u\op \phi(L_1^{d-1}\ot L_2) \op \phi(L_1^{d-1}\ot L_3)\op \phi(L_1^{s-2}\ot L_2\ot L_3)
$$
where $\phi: (L_1^{d-1}\ot L_x) \ra S^sA\ot S^tB\ot S^uC$ is the symmetric embedding.
The other two choices arise from exchanging the role of $L_1$ with $L_2,L_3$.
\end{theorem}
\begin{proof}
We have $T_{cw,2}(C^*)=\langle a_i\ot b_j+b_j\ot a_i \mid i\neq j\rangle$.
The only possibilites for $E_{110}$ for $r=4$ that pass the $(210)$-test
 arise by adding $a_k\ot b_k$ to this for some $k\in \{1,2,3\}$.
Take $k=1$. Then  
$$(E_{110}\ot A)\cap (S^2A\ot B)=
\langle
a_1^2\ot b_1, a_1a_2\ot b_1+a_1^2\ot b_2, a_1a_3\ot b_1+a_1^2\ot b_3, 
\sum_{\s\in \FS_3}a_{\s(1)}\ot a_{\s(2)}\ot b_{\s(e)}\rangle
$$
The only compatible choice of $E_{200}$ is $\langle a_1^2,a_1a_2,a_1a_3,a_2a_3\rangle$.
The situation for higher multi-degrees is similar.
\end{proof}

\begin{remark} In contrast to $T_{cw,2}$, 
by Corollary \ref{r5c3}, $\tdim \ul{VSP}(T_{skewcw,2})\geq 8$.
From \cite{MR3494510} (slightly changing notation)
we have the rank five decomposition:
\begin{align*}
T_{skewcw,2}=\frac 12[&
 2a_1\ot (b_2-b_3)\ot (c_2+c_3)\\
&-
(a_1+a_2)\ot (b_1-b_3)\ot (c_1+c_3)
-
(a_1-a_2)\ot (b_1+b_3)\ot (c_1-c_3)\\
&
+
(a_1+a_3)\ot (b_1-b_2)\ot (c_1+c_2)
-
(a_1-a_3)\ot (b_1+b_2)\ot (c_1-c_2)]
\end{align*}
and the orbit of this decomposition already
has dimension $8$. (This can be seen by noting that more
than four distinct vectors in $\BC^3$ appear in the decomposition.)
\end{remark}

\section{$T_{skewcw,q}$, $q>2$}

\begin{proof}[Proof of Theorem \ref{32qupper}]
For the upper bound, we have
 \begin{align}\label{brskewcw}
  T_{skewcw,q} =\tlim_{t\ra 0} 
&
\frac 1{t^3}[
\sum_{\xi}  [ (a_0+t^2a_{\xi})\ot (b_0-t^2b_{\xi})\ot (c_0-tc_{\xi+p})
+(a_0-t^2a_{\xi})\ot (b_0-tb_{\xi+p})\ot (c_0+t^2c_{\xi })\\
& \nonumber  \ \ \ \  \ 
+
(a_0-ta_{\xi+p})\ot (b_0+t^2b_{\xi})\ot (c_0-t^2c_{\xi })] \\
& \nonumber + 
\frac 1{t^5}(a_0+t^3 \sum_\xi a_{\xi+p})\ot (b_0+t^3 \sum_\xi b_{\xi+p})\ot (c_0+t^3 \sum_\xi c_{\xi+p}) \\
& \nonumber 
-(\frac{3q}{2t^2} +\frac 1{t^5})a_0\ot b_0\ot c_0].
\end{align}

For the lower bounds, 
write $A=B=C=L\op M$ with $\tdim L=1$, $\tdim M=q$ and $M$ is equipped with a symplectic form $\Omega$. We have the group $G=G_{T_{skewcw,q}}= Sp(M)\times GL(L)\times M^*\ot L\rtimes \BZ_3$.
Then
$$
A\ot B=L^{\ot 2} \op  L\cdot M\op   S^2 M\op \La 2 M_0\op (L \ww M \op \Omega)
$$
where the term  in parentheses equals $T_{skewcw,q}(C^*)$. 

we have the following weight diagram for
the $G_{T_{skewcw,q}}$-complement of $T(C^*)$ in $A\ot B$:
 
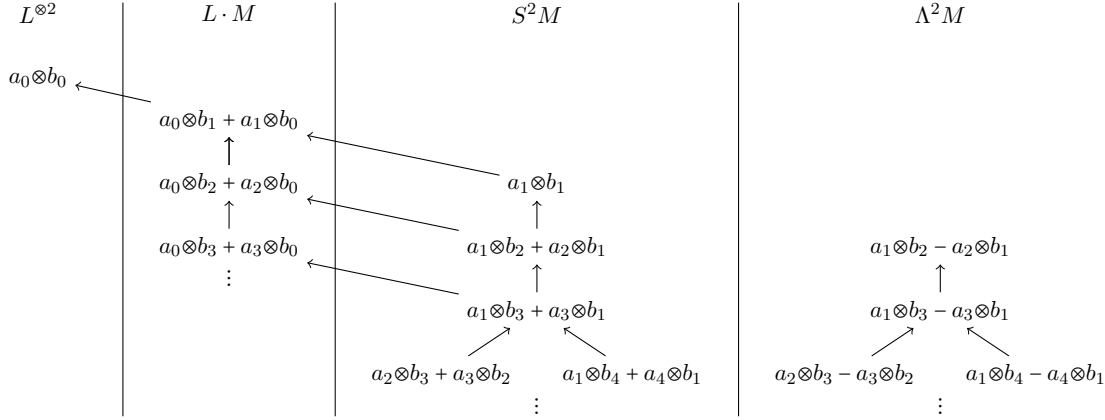
\begin{tikzpicture}[scale=0.8,every node/.style={scale=0.8}]
  \node at (0bp,0bp) {$L^{\ot 2}$};
  \node at (90bp,0bp) {$L\cdot M$};
  \node at (235bp,0bp) {$S^2 M$};
  \node at (425bp,0bp) {$\Lambda^2 M$};

  \node (b01) at (0bp,-30bp) {$a_0\ot b_0$};

  \node (b12) at (90bp,-50bp) {$a_0\ot b_1+a_1\ot b_0$};
  \node (b13) at (90bp,-80bp) {$a_0\ot b_2+a_2\ot b_0$};
  \node (b14) at (90bp,-110bp) {$a_0\ot b_3+a_3\ot b_0$};
  \node (b15) at (90bp,-125bp) {$\vdots$};

  \node (b23) at (235bp,-80bp) {$a_1\ot b_1$};
  \node (b24) at (235bp,-110bp) {$a_1\ot b_2+a_2\ot b_1$};
  \node (b25) at (235bp,-140bp) {$a_1\ot b_3 + a_3\ot b_1$};
  \node (b26) at (190bp,-170bp) {$a_2\ot b_3 + a_3\ot b_2$};
  \node (b26p) at (280bp,-170bp) {$a_1\ot b_4 +a_4\ot b_1$};
  \node (b27) at (235bp,-185bp) {$\vdots$};

  \node (b31) at (425bp,-110bp) {$a_1\ot b_2-a_2\ot b_1$};
  \node (b32) at (425bp,-140bp) {$a_1\ot b_3-a_3\ot b_1$};
  \node (b33) at (380bp,-170bp) {$a_2\ot b_3-a_3\ot b_2$};
  \node (b33p) at (470bp,-170bp) {$a_1\ot b_4-a_4\ot b_1$};
  \node (b34) at (425bp,-185bp) {$\vdots$};

  \draw (40bp,5bp) -- (40bp,-190bp);
  \draw (140bp,5bp) -- (140bp,-190bp);
  \draw (330bp,5bp) -- (330bp,-190bp);

  \draw[->] (b12) -> (b01);
  \draw[->] (b13) -> (b12);
  \draw[->] (b14) -> (b13);
  \draw[->] (b13) -> (b12);
  \draw[->] (b23) -> (b12);
  \draw[->] (b24) -> (b13);
  \draw[->] (b25) -> (b14);
  \draw[->] (b24) -> (b23);
  \draw[->] (b25) -> (b24);
  \draw[->] (b26) -> (b25);
  \draw[->] (b26p) -> (b25);
  \draw[->] (b32) -> (b31);
  \draw[->] (b33) -> (b32);
  \draw[->] (b33p) -> (b32);
\end{tikzpicture}

We will show that for $q\leq 10$, there is no choice of $E_{110}'$
satisfying all degree three tests when $r=\frac 32 q+1$. We focus on the case $q=10$ as
that is the most difficult, the other cases are easier.

Note that elements of $M$ may be raised to $L$,
so an element of $S^2M$ cannot be placed in $E_{110}'$ unless
its raising to $L\cdot M$ is also there. On the other hand, since
$L\ww M\subset E_{110}$, there is no similar restriction on elements
of $\La 2 M$.

We now restrict to  $q=10$.
We split the types of $(110)$ spaces into $10$ types of cases
depending on the dimension of $E_{110}'$ intersected with
the various irreducible modules:
$$
\begin{matrix}
{\rm case}& L^{\ot 2} & L\cdot M & S^2 M & \Lambda^2_0 M\\
1 & 1 & 4&0&0\\
2&  1 & 3&1&0\\
3&  1 & 2&2&0\\
4x&  1 & 2&1+\frac 12&\frac 12\\
5&  0 & 0&0&5\\
6&  1 & 0&0&4\\
7&  1 & 1&0&3\\
8&  1 & 2&0&2\\
9&  1 & 1&1&2\\
10&  1 & 2&1&1\end{matrix}
$$
Types $1,2,3,8,9,10$ are all single cases
Types  $5,6,7$ each involve a     choice of subset of weight vectors in $\Lambda^2_0 M$
(so they are each a collection of a finite number of cases)
and case $4$ involves a parameter, where we use $\frac 12$ to indicate the parameter,
as the weight vector is a sum of a vector in the two indicated spaces.
Explicitly, case $4x$ may be written
$$E_{110}'=
\langle a_0\ot b_0, a_0\ot b_1+a_1\ot b_0, a_0\ot b_2+a_2\ot b_0,
a_1\ot b_1, x(a_1\ot b_2+a_2\ot b_1)+ a_1\ot b_2-a_2\ot b \rangle.
$$
Of these cases $1,2,3,4x,8,10$ pass the $(210)$ and $(120)$ tests.
No triple passes the $(111)$ test. 
 \end{proof}

  We remark that the  decomposition \eqref{brskewcw} is $\BZ_3$-invariant.

 \begin{corollary}\label{vspskewcwq}
For $10\geq q>2$, and $q=2p$ even,   $\ul{VSP}(T_{skewcw,q})$   contains   the
   isotropic Grassmannian  $G_{\Omega}(\frac q2,M)$. In particular it has dimension at least $\binom{p}2$. 
\end{corollary}
\begin{proof}
Examining \eqref{brskewcw}, by $Sp(M)\subset G_{T_{skewcw,q}}$ we may  replace
$\langle a_{\xi}\rangle$ with any isotropic subspace as long as we replace
$\langle a_{\xi+p}\rangle$ with the corresponding dual subspace and
the same changes in $B,C$.
\end{proof}

\section{A simpler Waring border rank $17$ expression for $\tdet_3$}\label{det3bis}
Set $i=\sqrt{-1}$ and   $\zeta = e^{2\pi i/12}$. Then $\text{det}_3 = \sum_{s=1}^{17} m_s^{\ot 3}(t) + O(t)$,
where the $m_s$ are the following matrices
\begin{gather*}
\begin{pmatrix}
\frac{\zeta^{6}}{t^{5}} & 0 & 0 \\
0 & \zeta^{6} & 0 \\
0 & 0 & t^{5}
\end{pmatrix} \quad
\begin{pmatrix}
\frac{1}{t^{5}} & 0 & 0 \\
0 & 1 & 0 \\
0 & 0 & 0
\end{pmatrix} \quad
\begin{pmatrix}
\frac{\zeta^{6}}{t^{5}} & 0 & 0 \\
0 & 0 & t \zeta^{8} \\
0 & t^{4} \zeta^{4} & 0
\end{pmatrix} \quad
\begin{pmatrix}
\frac{\zeta^{4}}{t^{5}} & 0 & 0 \\
0 & 0 & t \zeta^{6} \\
0 & 0 & 0
\end{pmatrix}\\
\begin{pmatrix}
\frac{\zeta^{5}}{t^{5}} & 0 & 0 \\
0 & 0 & 0 \\
0 & t^{4} & 0
\end{pmatrix} \quad
\begin{pmatrix}
\frac{\zeta^{3}}{t^{5}} & 0 & 0 \\
0 & 0 & 0 \\
0 & 0 & t^{5}
\end{pmatrix} \quad
\begin{pmatrix}
0 & \frac{\zeta^{10}}{t^{4}} & \frac{\zeta^{8}}{t^{3}} \\
0 & 0 & t \zeta^{8} \\
t^{3} \zeta^{6} & 0 & 0
\end{pmatrix} \quad
\begin{pmatrix}
0 & \frac{\zeta^{8}}{t^{4}} & \frac{\zeta^{6}}{t^{3}} \\
0 & 0 & t \zeta^{6} \\
0 & 0 & 0
\end{pmatrix}\\
\begin{pmatrix}
0 & \frac{1}{t^{4}} & \frac{1}{t^{3}} \\
0 & 0 & 0 \\
t^{3} \zeta^{6} & 0 & 0
\end{pmatrix} \quad
\begin{pmatrix}
0 & \frac{\zeta^{6}}{t^{4}} & 0 \\
\frac{\zeta^{6}}{t} & 0 & 0 \\
0 & 0 & t^{5} \zeta^{6}
\end{pmatrix} \quad
\begin{pmatrix}
0 & \frac{\zeta^{11}}{t^{4}} & 0 \\
0 & 0 & 0 \\
t^{3} & 0 & 0
\end{pmatrix} \quad
\begin{pmatrix}
0 & \frac{\zeta^{9}}{t^{4}} & 0 \\
0 & 0 & 0 \\
0 & 0 & t^{5} \zeta^{6}
\end{pmatrix}\\
\begin{pmatrix}
0 & 0 & \frac{1}{t^{3}} \\
\frac{\zeta^{6}}{t} & 0 & 0 \\
0 & t^{4} \zeta^{6} & 0
\end{pmatrix} \quad
\begin{pmatrix}
0 & 0 & \frac{1}{t^{3}} \\
0 & \zeta^{4} & 0 \\
t^{3} \zeta^{2} & t^{4} & 0
\end{pmatrix} \quad
\begin{pmatrix}
0 & 0 & \frac{\zeta^{6}}{t^{3}} \\
0 & \zeta^{10} & 0 \\
0 & 0 & 0
\end{pmatrix} \\
\begin{pmatrix}
0 & \frac{5^{\frac{5}{6}} \zeta^{2}}{5  t^{4}} & \frac{ {\left(1+\frac{2}{5}  \sqrt{5}\right)}^{\frac{1}{3}} \zeta^{2}}{t^{3}} \\
\frac{{\left(1-\frac{2}{5}  \sqrt{5} \right)}^{\frac{1}{3}} \zeta^{8} }{t} & 0 & 0 \\
0 & 0 & 0
\end{pmatrix} \quad
\begin{pmatrix}
0 & \frac{5^{\frac{5}{6}} \zeta^{8}}{5  t^{4}} & \frac{{\left(1-\frac{2}{5}  \sqrt{5} \right)}^{\frac{1}{3}} \zeta^{2} }{t^{3}} \\
\frac{{\left(1+ \frac{2}{5}  \sqrt{5}\right)}^{\frac{1}{3}} \zeta^{8} }{t} & 0 & 0 \\
0 & 0 & 0
\end{pmatrix}.
\end{gather*}

\section{$T_{skewcw,4}^{\boxtimes 2}$}\label{skewcw4sect}
What follows is an expression for $T_{\text{skewcw},4}^{\boxtimes 2}$
as $\sum_{s=1}^{42} m_s(t)^{\ot 3} + O(t)$
that is    satisfied to an error of at most $4.4 \times  10^{-15}$ in each entry. 
It consists of $42$ matrices whose entries are rational expressions in
the following $36$ complex numbers:
Let $i=\sqrt{-1}$ and let $\zeta = e^{2\pi i/ 12}$. Set
\begin{tiny}
\begin{align*}
  z_{0} &= -0.8660155098072051 + 0.9452855522785384i &
  z_{1} &= -1.2981710770246242 + 0.0008968724089185688i \\
  z_{2} &= 2.9260271139931078 + 0.1853833642730014i &
  z_{3} &= 0.2542517122150322 + 0.30793819438378284i \\
  z_{4} &= 0.6964375578992822 + 0.2772662627986198i &
  z_{5} &= 0.5507020325318998 - 0.0493931308002328i \\
  z_{6} &= 1.149228383831849 - 1.1683147648642283i &
  z_{7} &= 0.6586404058476252 - 0.16578044112199047i \\
  z_{8} &= 0.7654345273805864 - 0.06877274843008892i &
  z_{9} &= 0.544690883860558 + 0.09720573163212605i \\
  z_{10} &= 0.6932236636741451 + 0.14980159446358277i &
  z_{11} &= 0.5862637032385472 - 0.12844523449559558i \\
  z_{12} &= 2.384363992555291 - 0.08927102369428247i &
  z_{13} &= 0.9664252976479286 + 0.08480470055107503i \\
  z_{14} &= 0.6190926897383283 + 0.15631000400545272i &
  z_{15} &= 0.6283592253932955 - 0.5626050553495663i \\
  z_{16} &= 1.8190778570602204 - 0.22163457440913656i &
  z_{17} &= 1.153187286528645 - 0.07977233251120702i \\
  z_{18} &= 1.4498877801613976 - 0.22515738202335905i &
  z_{19} &= 0.7262464450114047 + 0.7050051641972112i \\
  z_{20} &= 1.1195537528292199 - 0.26381000320340176i &
  z_{21} &= 0.4400325048210471 + 0.6593492930106759i \\
  z_{22} &= 0.3476654993676339 + 0.4095417606798612i &
  z_{23} &= 0.9459769225333798 + 0.24589162882727128i \\
  z_{24} &= 0.7637135867709066 - 0.10529269213820387i &
  z_{25} &= 0.7409392923310902 - 0.10474756303325146i \\
  z_{26} &= 1.0112068238001992 - 0.12695675940574122i &
  z_{27} &= 1.5005677845016696 - 0.24533651960180036i \\
  z_{28} &= 0.6134145054919202 + 0.08121891266185506i &
  z_{29} &= 1.145625294745251 - 0.3813562005184122i \\
  z_{30} &= 1.0607612533915372 - 0.016294891090460426i &
  z_{31} &= 0.941339345482511 + 0.20413704882122435i \\
  z_{32} &= 0.622575977639622 + 0.2555810563389569i &
  z_{33} &= 0.951746321194872 - 0.2894768358835511i \\
  z_{34} &= 1.0532801812660977 - 0.2502246606675517i &
  z_{35} &= 1.0207644184200035 - 0.2106937666100475i.
\end{align*}
\end{tiny}
  The $42$   matrices are:
 \begin{tiny}
\[
\begin{pmatrix}
\frac{\zeta^{3} z_{25} z_{26} z_{31} z_{35}}{t^{269}} & \frac{\zeta^{8} z_{28} t^{61}}{z_{23} z_{25} z_{26} z_{31} z_{33} z_{34}} & \frac{\zeta^{10} z_{24} z_{35} t^{13}}{z_{23} z_{25} z_{26} z_{31} z_{33} z_{34}} & 0 & 0\\
\frac{z_{25} z_{26} z_{31} z_{33} z_{35}}{t^{104}} & 0 & 0 & 0 & \frac{\zeta^{9} z_{18} z_{26} z_{32} t^{148}}{z_{21} z_{23} z_{33} z_{35}^{2}}\\
0 & 0 & 0 & 0 & 0\\
0 & 0 & \frac{\zeta^{7} z_{35} t^{121}}{z_{23} z_{25} z_{26}^{2} z_{31} z_{33} z_{34}} & 0 & \frac{\zeta^{8} z_{23} z_{34} t^{91}}{z_{24} z_{35}^{2}}\\
0 & 0 & 0 & 0 & 0
\end{pmatrix}
\ \ 
\begin{pmatrix}
\frac{\zeta^{10} z_{30}}{t^{269}} & 0 & \frac{\zeta^{10} z_{24} z_{35} t^{13}}{z_{27} z_{28} z_{30} z_{33}} & 0 & 0\\
0 & 0 & \frac{\zeta^{7} z_{24} z_{35} t^{178}}{z_{27} z_{28} z_{30}} & 0 & 0\\
0 & 0 & 0 & 0 & 0\\
0 & 0 & \frac{\zeta^{6} z_{27} z_{28} z_{35}^{2} t^{121}}{z_{23}^{2} z_{25} z_{26}^{2} z_{31} z_{33} z_{34}^{2}} & 0 & 0\\
0 & \frac{\zeta t^{184}}{z_{27} z_{30} z_{34}} & 0 & 0 & 0
\end{pmatrix}
\]
\[
\begin{pmatrix}
0 & 0 & 0 & 0 & 0\\
0 & 0 & 0 & 0 & 0\\
0 & \frac{z_{18} z_{24} z_{28} z_{32} t^{211}}{z_{21} z_{23}^{2} z_{25} z_{33}^{3} z_{35}} & z_{0} t^{163} & 0 & \frac{\zeta z_{34} t^{133}}{z_{25}}\\
0 & 0 & 0 & 0 & 0\\
\frac{\zeta z_{23} z_{25} z_{33}}{z_{24} z_{31} t^{146}} & \frac{\zeta^{5} z_{31} z_{35} t^{184}}{z_{23} z_{27}^{2} z_{28} z_{30} z_{34}^{2}} & 0 & 0 & 0
\end{pmatrix}
\ \ 
\begin{pmatrix}
\frac{\zeta^{9} z_{30}}{z_{27} t^{269}} & 0 & 0 & \frac{\zeta^{7}}{t^{65}} & 0\\
0 & 0 & 0 & 0 & 0\\
0 & 0 & 0 & 0 & 0\\
0 & \frac{\zeta^{5} z_{27} t^{169}}{z_{30} z_{33}} & 0 & 0 & 0\\
0 & \frac{\zeta^{9} z_{27} t^{184}}{z_{30} z_{34}} & \frac{z_{28} t^{136}}{z_{26} z_{34} z_{35}} & 0 & 0
\end{pmatrix}
\]
\[
\begin{pmatrix}
\frac{\zeta^{4} z_{19} z_{35}}{z_{3} z_{12} z_{13} z_{18} z_{32} t^{269}} & 0 & 0 & 0 & \frac{\zeta^{4} z_{18} z_{21} z_{34}}{z_{2} z_{12} z_{25} z_{31} z_{32} z_{35} t^{17}}\\
\frac{\zeta z_{19} z_{33} z_{35}}{z_{3} z_{12} z_{13} z_{18} z_{32} t^{104}} & 0 & 0 & 0 & \frac{\zeta^{9} z_{19} z_{26} z_{32}^{2} z_{33} z_{35} t^{148}}{z_{3} z_{12} z_{13} z_{18}}\\
0 & 0 & 0 & 0 & 0\\
\frac{\zeta^{4} z_{2} z_{3} z_{13} z_{31}}{z_{12} z_{19} z_{20} z_{21} z_{32} z_{34} t^{161}} & 0 & 0 & 0 & 0\\
0 & 0 & 0 & 0 & 0
\end{pmatrix}
\ \ 
\begin{pmatrix}
\frac{1}{z_{34} t^{269}} & 0 & 0 & \frac{\zeta z_{24} z_{26} z_{35}}{z_{28} z_{34} t^{65}} & 0\\
0 & 0 & 0 & 0 & 0\\
\frac{\zeta^{5}}{t^{119}} & 0 & 0 & \frac{z_{24} z_{26} z_{35} t^{85}}{z_{28}} & 0\\
0 & 0 & 0 & 0 & 0\\
0 & 0 & 0 & 0 & 0
\end{pmatrix}
\]
\[
\begin{pmatrix}
0 & \frac{\zeta^{2} z_{15} z_{18} z_{19} z_{28} z_{32} t^{61}}{z_{21} z_{23} z_{25} z_{31} z_{33}^{2} z_{34}} & \frac{\zeta^{11} z_{15} z_{19} z_{35} t^{13}}{z_{34}} & 0 & 0\\
0 & 0 & 0 & 0 & 0\\
0 & 0 & 0 & 0 & 0\\
\frac{\zeta^{11} z_{15} z_{21} z_{23} z_{25} z_{31} z_{35}}{z_{18} z_{19}^{2} z_{26} z_{32} z_{34} t^{161}} & 0 & 0 & 0 & 0\\
\frac{\zeta^{9} z_{21} z_{23} z_{25} z_{33} z_{34}}{z_{15}^{2} z_{18} z_{19}^{2} z_{26} z_{31}^{2} z_{32} z_{35}^{2} t^{146}} & 0 & 0 & 0 & 0
\end{pmatrix}
\
\ 
\begin{pmatrix}
\frac{\zeta^{5} z_{25} z_{31}}{z_{23} z_{34} t^{269}} & 0 & 0 & 0 & \frac{\zeta^{4}}{t^{17}}\\
0 & 0 & 0 & 0 & 0\\
\frac{\zeta^{10} z_{25} z_{31}}{z_{23} t^{119}} & 0 & \frac{\zeta^{7} z_{23} z_{24} z_{33} t^{163}}{z_{20}^{2} z_{31}} & 0 & \zeta^{3} z_{34} t^{133}\\
0 & \frac{\zeta^{2} z_{23}^{2} z_{28} z_{34}^{2} t^{169}}{z_{25}^{2} z_{31}^{2} z_{33} z_{35}} & 0 & 0 & 0\\
0 & 0 & \frac{\zeta^{10} z_{23} t^{136}}{z_{25} z_{31}} & 0 & 0
\end{pmatrix}
\]
\[
\begin{pmatrix}
\frac{\zeta^{11} z_{3} z_{19}}{z_{12} z_{21} z_{23} z_{25} z_{32} t^{269}} & 0 & 0 & 0 & \frac{\zeta^{10} z_{2} z_{18} z_{21} z_{23} z_{24} z_{34}}{z_{12} z_{13} z_{19} z_{20} z_{31} z_{32} z_{35} t^{17}}\\
\frac{\zeta^{8} z_{3} z_{19} z_{33}}{z_{12} z_{21} z_{23} z_{25} z_{32} t^{104}} & 0 & 0 & 0 & \frac{\zeta^{4} z_{3} z_{19} z_{26} z_{32}^{2} z_{33} t^{148}}{z_{12} z_{21} z_{23} z_{25}}\\
0 & 0 & 0 & 0 & 0\\
\frac{\zeta^{9} z_{13} z_{31} z_{35}}{z_{2} z_{3} z_{12} z_{18} z_{24} z_{32} z_{34} t^{161}} & 0 & 0 & 0 & 0\\
0 & 0 & 0 & 0 & 0
\end{pmatrix}
\ \ 
\begin{pmatrix}
\frac{\zeta^{9} z_{28}}{t^{269}} & 0 & 0 & 0 & 0\\
0 & 0 & 0 & 0 & 0\\
\frac{\zeta^{2} z_{28} z_{34}}{t^{119}} & \frac{\zeta^{10} z_{34} t^{211}}{z_{28} z_{33} z_{35}} & 0 & 0 & 0\\
0 & 0 & 0 & 0 & 0\\
0 & 0 & 0 & \frac{\zeta^{5} z_{33} z_{35} t^{58}}{z_{34}} & 0
\end{pmatrix}
\]
\[
\begin{pmatrix}
\frac{\zeta^{5} z_{24} z_{26} z_{29}^{3} z_{30}}{z_{27} t^{269}} & 0 & \frac{\zeta^{11} z_{22}^{3} z_{27} z_{33} z_{35} t^{13}}{z_{20} z_{23} z_{28} z_{30} z_{31}^{2}} & \frac{\zeta^{2}}{t^{65}} & 0\\
0 & 0 & 0 & 0 & 0\\
\frac{\zeta^{10} z_{24} z_{26} z_{29}^{3} z_{30} z_{34}}{z_{27} t^{119}} & 0 & 0 & \zeta z_{34} t^{85} & \frac{\zeta^{10} z_{18} z_{22}^{3} z_{26} z_{27} z_{32} z_{34} z_{35} t^{133}}{z_{20} z_{21} z_{23}^{2} z_{25} z_{28} z_{30}}\\
0 & 0 & 0 & 0 & 0\\
\frac{\zeta^{10} z_{18} z_{22}^{3} z_{25} z_{26} z_{27} z_{32} z_{33} z_{35}}{z_{20} z_{21} z_{23} z_{24} z_{28} z_{30} z_{31} t^{146}} & \frac{z_{27} t^{184}}{z_{24} z_{26} z_{29}^{3} z_{30} z_{34}} & 0 & 0 & 0
\end{pmatrix}
\]
\[
\begin{pmatrix}
\frac{\zeta^{2} z_{25} z_{26} z_{31} z_{32} z_{35}}{z_{18} z_{21} z_{34} t^{269}} & \frac{\zeta^{11} z_{21} z_{24} z_{32}^{2} z_{35} t^{61}}{z_{25} z_{31}} & 0 & 0 & 0\\
\frac{\zeta^{11} z_{25} z_{26} z_{31} z_{32} z_{33} z_{35}}{z_{18} z_{21} z_{34} t^{104}} & 0 & 0 & 0 & \frac{\zeta^{11} z_{28} t^{148}}{z_{21}^{2} z_{23} z_{24} z_{32} z_{33} z_{34} z_{35}^{3}}\\
0 & 0 & 0 & 0 & 0\\
\frac{\zeta^{11} z_{25} z_{31} z_{32} z_{35}}{z_{18} z_{21} z_{24} z_{34} t^{161}} & 0 & 0 & 0 & 0\\
0 & 0 & 0 & 0 & 0
\end{pmatrix}
\ \ 
\begin{pmatrix}
0 & 0 & 0 & 0 & \frac{\zeta^{6} z_{18} z_{26} z_{31} z_{32}}{z_{20} z_{21} z_{34} t^{17}}\\
0 & 0 & 0 & 0 & 0\\
0 & 0 & 0 & 0 & 0\\
\frac{\zeta z_{21}^{2} z_{23}^{2} z_{24} z_{25} z_{33}}{z_{18}^{2} z_{19}^{3} z_{20} z_{26}^{2} z_{31} z_{32}^{2} t^{161}} & 0 & 0 & 0 & 0\\
\frac{\zeta^{11}}{t^{146}} & 0 & 0 & 0 & 0
\end{pmatrix}
\]
\[
\begin{pmatrix}
\frac{\zeta^{3} z_{9} z_{10} z_{11} z_{18} z_{19} z_{30} z_{32} z_{34}}{z_{5} z_{7} z_{8} z_{16} z_{21} z_{23} z_{25} z_{33} t^{269}} & 0 & 0 & 0 & \frac{\zeta z_{5} z_{11} z_{16} z_{18} z_{24} z_{26} z_{31} z_{32}}{z_{4} z_{20} z_{21} z_{25} t^{17}}\\
0 & 0 & \frac{\zeta^{11} z_{5} z_{11} z_{16} z_{20} z_{33} t^{178}}{z_{4}} & 0 & 0\\
0 & 0 & 0 & 0 & 0\\
\frac{\zeta^{10} z_{4}}{z_{19} z_{20} z_{27} z_{30} z_{33} t^{161}} & 0 & 0 & 0 & 0\\
\frac{\zeta^{2} z_{4}}{z_{19} z_{20} z_{27} z_{30} z_{34} t^{146}} & 0 & \frac{\zeta^{6} z_{9} z_{10} z_{11} z_{18} z_{19} z_{27} z_{28} z_{32} t^{136}}{z_{5} z_{7} z_{8} z_{16} z_{21} z_{23} z_{25} z_{26} z_{33} z_{35}} & 0 & 0
\end{pmatrix}
\ \ 
\begin{pmatrix}
\frac{\zeta^{5} z_{28} z_{35}^{2}}{t^{269}} & 0 & 0 & \frac{\zeta^{5}}{t^{65}} & 0\\
\frac{\zeta^{2} z_{28} z_{33} z_{35}^{2}}{t^{104}} & 0 & 0 & \zeta^{8} z_{33} t^{100} & 0\\
0 & 0 & 0 & 0 & 0\\
0 & 0 & 0 & t^{43} & 0\\
0 & 0 & 0 & \frac{\zeta^{10} z_{33} t^{58}}{z_{34}} & 0
\end{pmatrix}
\]
\[
\begin{pmatrix}
0 & \frac{\zeta^{9} z_{17}^{2} z_{21} z_{24} z_{28} z_{29}^{2} t^{61}}{z_{18} z_{22}^{2} z_{25} z_{27}^{2} z_{32} z_{33} z_{34} z_{35}} & \frac{\zeta^{4} z_{22} z_{27} t^{13}}{z_{17} z_{26} z_{29} z_{31}^{2}} & \frac{\zeta z_{22} z_{27} z_{31} z_{33} z_{35}}{z_{17} z_{28} z_{29} t^{65}} & 0\\
0 & 0 & 0 & 0 & 0\\
0 & \frac{\zeta^{2} z_{17}^{2} z_{21} z_{24} z_{28} z_{29}^{2} t^{211}}{z_{18} z_{22}^{2} z_{25} z_{27}^{2} z_{32} z_{33} z_{35}} & 0 & \frac{z_{22} z_{27} z_{31} z_{33} z_{34} z_{35} t^{85}}{z_{17} z_{28} z_{29}} & \frac{\zeta^{3} z_{18} z_{22} z_{27} z_{32} z_{34} t^{133}}{z_{17} z_{21} z_{23} z_{25} z_{29} z_{33}}\\
0 & 0 & 0 & 0 & 0\\
\frac{\zeta^{3} z_{18} z_{22} z_{25} z_{27} z_{32}}{z_{17} z_{21} z_{24} z_{29} z_{31} t^{146}} & \frac{\zeta^{2} z_{17}^{2} z_{22} t^{184}}{z_{20} z_{23} z_{24} z_{29} z_{30}^{2} z_{31} z_{34}} & 0 & 0 & 0
\end{pmatrix}
\]
\[
\begin{pmatrix}
\frac{\zeta^{11} z_{27} z_{28} z_{30}^{2}}{z_{26} z_{35} t^{269}} & 0 & 0 & 0 & \frac{\zeta^{3}}{t^{17}}\\
0 & 0 & \frac{\zeta z_{20}^{2} z_{21} z_{25} z_{33} t^{178}}{z_{18} z_{24} z_{26} z_{31} z_{32}} & 0 & 0\\
0 & 0 & 0 & 0 & 0\\
0 & 0 & \frac{\zeta^{7} z_{26} z_{35} t^{121}}{z_{27} z_{28} z_{30}^{2} z_{33}} & 0 & 0\\
0 & 0 & \frac{z_{25} z_{31} t^{136}}{z_{23} z_{26} z_{34}^{2}} & 0 & 0
\end{pmatrix}
\ \ 
\begin{pmatrix}
\frac{\zeta^{8} z_{25} z_{26} z_{31} z_{32} z_{35}}{z_{18} z_{21} z_{34} t^{269}} & 0 & \frac{\zeta^{4} z_{21} z_{24} z_{35} t^{13}}{z_{25} z_{26} z_{31} z_{32}} & 0 & 0\\
\frac{\zeta^{5} z_{25} z_{26} z_{31} z_{32} z_{33} z_{35}}{z_{18} z_{21} z_{34} t^{104}} & 0 & 0 & 0 & \frac{\zeta^{5} z_{28} t^{148}}{z_{21}^{2} z_{23} z_{24} z_{32} z_{33} z_{34} z_{35}^{3}}\\
0 & 0 & 0 & 0 & 0\\
\frac{\zeta^{5} z_{25} z_{31} z_{32} z_{35}}{z_{18} z_{21} z_{24} z_{34} t^{161}} & 0 & 0 & 0 & 0\\
0 & 0 & 0 & 0 & 0
\end{pmatrix}
\]
\[
\begin{pmatrix}
\frac{\zeta^{10} z_{5} z_{18} z_{19} z_{28} z_{30} z_{32}}{z_{21} z_{23} z_{25} z_{33} z_{35} t^{269}} & 0 & 0 & 0 & \frac{z_{4} z_{11} z_{16} z_{31}}{z_{5} z_{19} z_{20} t^{17}}\\
0 & 0 & \frac{\zeta^{10} z_{4} z_{11} z_{16} z_{20} z_{21} z_{25} z_{33} t^{178}}{z_{5} z_{18} z_{19} z_{24} z_{26} z_{32}} & 0 & 0\\
0 & 0 & 0 & 0 & 0\\
\frac{\zeta^{10} z_{9} z_{10} z_{11} z_{18} z_{24} z_{26} z_{32} z_{34} z_{35}}{z_{4} z_{7} z_{8} z_{16} z_{20} z_{21} z_{25} z_{27} z_{28} z_{30} z_{33} t^{161}} & 0 & 0 & 0 & 0\\
\frac{\zeta^{2} z_{9} z_{10} z_{11} z_{18} z_{24} z_{26} z_{32} z_{35}}{z_{4} z_{7} z_{8} z_{16} z_{20} z_{21} z_{25} z_{27} z_{28} z_{30} t^{146}} & 0 & \frac{\zeta z_{5} z_{18} z_{19} z_{27} z_{28}^{2} z_{32} t^{136}}{z_{21} z_{23} z_{25} z_{26} z_{33} z_{34} z_{35}^{2}} & 0 & 0
\end{pmatrix}
\]
\[
\begin{pmatrix}
\frac{\zeta^{4} z_{17}^{2} z_{22} z_{26} z_{29}}{z_{20} z_{23} z_{30} z_{34} t^{269}} & 0 & \frac{z_{22} z_{27} z_{28} z_{29} z_{30} z_{34}^{2} t^{13}}{z_{17} z_{26} z_{31} z_{35}} & \frac{\zeta^{8} z_{22} z_{24} z_{26} z_{29} z_{33} z_{35}^{2}}{z_{17} z_{27} z_{28}^{2} z_{31} z_{34} t^{65}} & 0\\
0 & 0 & 0 & 0 & 0\\
\frac{\zeta^{9} z_{17}^{2} z_{22} z_{26} z_{29}}{z_{20} z_{23} z_{30} t^{119}} & 0 & 0 & \frac{\zeta^{7} z_{22} z_{24} z_{26} z_{29} z_{33} z_{35}^{2} t^{85}}{z_{17} z_{27} z_{28}^{2} z_{31}} & \frac{\zeta^{11} z_{18} z_{22} z_{27} z_{28} z_{29} z_{30} z_{31} z_{32} z_{34}^{3} t^{133}}{z_{17} z_{21} z_{23} z_{25} z_{33} z_{35}}\\
0 & 0 & 0 & 0 & 0\\
\frac{\zeta^{11} z_{18} z_{22} z_{25} z_{27} z_{28} z_{29} z_{30} z_{32} z_{34}^{2}}{z_{17} z_{21} z_{24} z_{35} t^{146}} & 0 & 0 & 0 & 0
\end{pmatrix}
\]
\[
\begin{pmatrix}
0 & 0 & \zeta^{5} z_{19} z_{31} t^{13} & 0 & 0\\
0 & 0 & \zeta^{8} z_{19} z_{31} z_{33} t^{178} & 0 & 0\\
0 & 0 & 0 & 0 & 0\\
\frac{\zeta^{8} z_{21} z_{23} z_{25}}{z_{18} z_{19}^{2} z_{26} z_{31} z_{32} t^{161}} & 0 & 0 & 0 & 0\\
\frac{z_{21} z_{23} z_{25} z_{33}}{z_{18} z_{19}^{2} z_{26} z_{31} z_{32} z_{34} t^{146}} & 0 & 0 & 0 & 0
\end{pmatrix}
\ \ 
\begin{pmatrix}
\frac{\zeta^{3}}{t^{269}} & \frac{\zeta^{7} t^{61}}{z_{33}} & 0 & \frac{\zeta^{10}}{t^{65}} & 0\\
\frac{z_{33}}{t^{104}} & 0 & 0 & \zeta z_{33} t^{100} & 0\\
0 & 0 & 0 & 0 & 0\\
\frac{\zeta^{4}}{t^{161}} & 0 & 0 & 0 & 0\\
0 & 0 & 0 & 0 & 0
\end{pmatrix}
\]
\[
\begin{pmatrix}
0 & \frac{\zeta^{10} z_{16}^{2} z_{21} z_{23} z_{25} z_{28} t^{61}}{z_{18} z_{19} z_{32} z_{35}} & 0 & \frac{\zeta^{4} z_{18} z_{19}^{2} z_{26} z_{31} z_{32} z_{35}}{z_{16} z_{21} z_{23} z_{25} z_{28} t^{65}} & 0\\
0 & 0 & 0 & 0 & 0\\
0 & 0 & 0 & 0 & 0\\
\frac{\zeta}{z_{16} z_{19} z_{26} z_{31} z_{33} t^{161}} & 0 & 0 & 0 & 0\\
\frac{\zeta^{5}}{z_{16} z_{19} z_{26} z_{31} z_{34} t^{146}} & 0 & 0 & 0 & 0
\end{pmatrix}
\ \ 
\begin{pmatrix}
\frac{\zeta^{7} z_{28}}{t^{269}} & \frac{\zeta^{3} t^{61}}{z_{28} z_{33} z_{35}} & 0 & 0 & 0\\
0 & 0 & 0 & 0 & 0\\
\frac{z_{28} z_{34}}{t^{119}} & 0 & 0 & 0 & 0\\
0 & \frac{\zeta^{4} t^{169}}{z_{28} z_{33} z_{35}} & 0 & 0 & 0\\
0 & 0 & 0 & \frac{\zeta^{3} z_{33} z_{35} t^{58}}{z_{34}} & 0
\end{pmatrix}
\]
\[
\begin{pmatrix}
\frac{z_{14}}{t^{269}} & \frac{\zeta^{4} z_{30} t^{61}}{z_{14}^{2} z_{27} z_{33}} & 0 & 0 & 0\\
0 & 0 & 0 & 0 & 0\\
0 & 0 & 0 & 0 & 0\\
0 & 0 & 0 & \frac{\zeta^{11} z_{14} z_{27} t^{43}}{z_{30}} & 0\\
0 & 0 & 0 & 0 & 0
\end{pmatrix}
\ \ 
\begin{pmatrix}
\frac{\zeta^{10} z_{30}}{t^{269}} & \frac{\zeta^{2} t^{61}}{z_{27} z_{30} z_{33}} & 0 & 0 & 0\\
0 & 0 & 0 & 0 & 0\\
0 & 0 & 0 & 0 & 0\\
0 & 0 & \frac{\zeta z_{35} t^{121}}{z_{26} z_{27} z_{28} z_{30} z_{33}} & 0 & \frac{\zeta z_{27} z_{28} t^{91}}{z_{24} z_{35}}\\
0 & 0 & 0 & 0 & \frac{\zeta^{5} z_{27} z_{28} z_{33} t^{106}}{z_{24} z_{34} z_{35}}
\end{pmatrix}
\]
\[
\begin{pmatrix}
\frac{\zeta^{4}}{t^{269}} & \frac{\zeta^{3} z_{28} t^{61}}{z_{33} z_{35}} & 0 & 0 & 0\\
0 & 0 & 0 & 0 & 0\\
\frac{\zeta^{9} z_{34}}{t^{119}} & \frac{\zeta^{2} z_{28} z_{34} t^{211}}{z_{33} z_{35}} & 0 & 0 & 0\\
0 & 0 & \frac{\zeta^{2} t^{121}}{z_{26} z_{33}} & 0 & 0\\
0 & 0 & 0 & 0 & \frac{\zeta^{4} z_{33} t^{106}}{z_{24} z_{34}}
\end{pmatrix}
\ \ 
\begin{pmatrix}
\frac{\zeta^{2} z_{14}}{t^{269}} & 0 & 0 & 0 & 0\\
0 & \frac{\zeta^{9} z_{30} t^{226}}{z_{14}^{2} z_{27}} & 0 & 0 & 0\\
0 & 0 & 0 & 0 & 0\\
0 & \frac{\zeta z_{30} t^{169}}{z_{14}^{2} z_{27} z_{33}} & 0 & \frac{\zeta z_{14} z_{27} t^{43}}{z_{30}} & 0\\
0 & 0 & 0 & 0 & 0
\end{pmatrix}
\]
\[
\begin{pmatrix}
\frac{\zeta^{4} z_{30}}{t^{269}} & 0 & \frac{\zeta^{4} z_{24} z_{35} t^{13}}{z_{27} z_{28} z_{30} z_{33}} & 0 & 0\\
0 & 0 & \frac{\zeta z_{24} z_{35} t^{178}}{z_{27} z_{28} z_{30}} & 0 & 0\\
0 & 0 & 0 & 0 & 0\\
0 & 0 & 0 & 0 & \frac{\zeta^{7} z_{27} z_{28} t^{91}}{z_{24} z_{35}}\\
0 & 0 & 0 & 0 & \frac{\zeta^{11} z_{27} z_{28} z_{33} t^{106}}{z_{24} z_{34} z_{35}}
\end{pmatrix}
\ \ 
\begin{pmatrix}
\frac{\zeta^{8} z_{18} z_{22} z_{24} z_{26} z_{27} z_{32}}{z_{20} z_{21} z_{23} z_{31} z_{33} t^{269}} & 0 & \frac{\zeta^{2} z_{20} z_{22} z_{27} z_{34} t^{13}}{z_{33}} & 0 & 0\\
0 & 0 & 0 & 0 & 0\\
\frac{\zeta z_{18} z_{22} z_{24} z_{26} z_{27} z_{32} z_{34}}{z_{20} z_{21} z_{23} z_{31} z_{33} t^{119}} & 0 & \frac{\zeta^{3} z_{21} z_{23} z_{24} z_{31} t^{163}}{z_{18} z_{22}^{2} z_{26} z_{27}^{2} z_{32}} & 0 & \frac{\zeta^{10} z_{22} z_{27} z_{31} z_{33} t^{133}}{z_{20} z_{23} z_{25}}\\
0 & 0 & 0 & 0 & 0\\
\frac{\zeta^{10} z_{22} z_{25} z_{27} z_{33}^{2}}{z_{20} z_{24} z_{34} t^{146}} & \frac{\zeta^{3} z_{20} z_{22} z_{35} t^{184}}{z_{24} z_{27} z_{28} z_{30}} & 0 & 0 & 0
\end{pmatrix}
\]
\[
\begin{pmatrix}
\frac{\zeta^{10}}{t^{269}} & 0 & 0 & 0 & 0\\
0 & 0 & 0 & 0 & 0\\
\frac{\zeta^{3} z_{34}}{t^{119}} & 0 & 0 & 0 & 0\\
0 & 0 & \frac{\zeta^{8} t^{121}}{z_{26} z_{33}} & 0 & 0\\
0 & z_{1} t^{184} & \frac{\zeta^{6} z_{25} z_{31} t^{136}}{z_{23} z_{34}^{2}} & 0 & 0
\end{pmatrix}
\ \ 
\begin{pmatrix}
\frac{\zeta^{5} z_{24} z_{25} z_{33}}{z_{20}^{2} z_{34} t^{269}} & 0 & 0 & 0 & \frac{\zeta^{8}}{t^{17}}\\
0 & 0 & 0 & 0 & 0\\
\frac{\zeta^{10} z_{24} z_{25} z_{33}}{z_{20}^{2} t^{119}} & 0 & \zeta^{9} t^{163} & 0 & \frac{\zeta^{3} z_{23} z_{24} z_{33} z_{34} t^{133}}{z_{20}^{2} z_{31}}\\
0 & 0 & 0 & 0 & 0\\
\frac{\zeta}{t^{146}} & 0 & 0 & 0 & 0
\end{pmatrix}
\]
\[
\begin{pmatrix}
\frac{\zeta^{6}}{t^{269}} & 0 & \frac{\zeta z_{24} t^{13}}{z_{33}} & 0 & 0\\
0 & 0 & 0 & 0 & 0\\
\frac{\zeta^{11} z_{34}}{t^{119}} & 0 & \frac{z_{24} z_{34} t^{163}}{z_{33}} & 0 & 0\\
0 & 0 & 0 & 0 & 0\\
0 & 0 & 0 & 0 & \frac{\zeta^{6} z_{33} t^{106}}{z_{24} z_{34}}
\end{pmatrix}
\ \ 
\begin{pmatrix}
\frac{\zeta^{10} z_{9} z_{34}}{t^{269}} & 0 & 0 & 0 & 0\\
0 & \frac{\zeta z_{7} z_{8} z_{11} z_{18} z_{31} z_{32} z_{35} t^{226}}{z_{10} z_{21} z_{23} z_{25} z_{27} z_{28} z_{30}} & 0 & 0 & 0\\
0 & 0 & 0 & 0 & 0\\
\frac{\zeta^{5} z_{9} z_{34}}{t^{161}} & 0 & 0 & 0 & 0\\
\frac{\zeta^{9} z_{9} z_{33}}{t^{146}} & 0 & \frac{\zeta z_{9} z_{27} z_{28} t^{136}}{z_{26} z_{30} z_{35}} & 0 & 0
\end{pmatrix}
\]
\[
\begin{pmatrix}
\frac{\zeta^{7} z_{21} z_{24}}{z_{18} z_{23} z_{26} z_{32} t^{269}} & 0 & 0 & 0 & \frac{\zeta^{5}}{t^{17}}\\
0 & 0 & 0 & 0 & 0\\
\frac{z_{21} z_{24} z_{34}}{z_{18} z_{23} z_{26} z_{32} t^{119}} & 0 & 0 & 0 & \frac{\zeta^{5} z_{21} z_{24} z_{34}^{2} t^{133}}{z_{18} z_{25} z_{26} z_{31} z_{32}}\\
0 & 0 & 0 & 0 & 0\\
\frac{\zeta^{8} z_{18} z_{23} z_{25} z_{26} z_{32} z_{33}}{z_{20}^{2} z_{21} z_{34} t^{146}} & 0 & 0 & 0 & 0
\end{pmatrix}
\ \ 
\begin{pmatrix}
\frac{\zeta^{9}}{t^{269}} & 0 & 0 & \frac{\zeta^{4}}{t^{65}} & 0\\
\frac{\zeta^{6} z_{33}}{t^{104}} & 0 & 0 & \zeta^{7} z_{33} t^{100} & 0\\
0 & 0 & 0 & 0 & 0\\
\frac{\zeta^{10}}{t^{161}} & 0 & 0 & 0 & 0\\
0 & 0 & 0 & 0 & 0
\end{pmatrix}
\]
\[
\begin{pmatrix}
\frac{\zeta^{11} z_{25} z_{35}}{z_{13} t^{269}} & 0 & 0 & 0 & \frac{\zeta^{7} z_{23} z_{34}}{z_{13} z_{31} z_{35}^{2} t^{17}}\\
\frac{\zeta^{8} z_{25} z_{33} z_{35}}{z_{13} t^{104}} & 0 & 0 & 0 & \frac{\zeta^{4} z_{25} z_{26} z_{32}^{3} z_{33} z_{35} t^{148}}{z_{13}}\\
0 & 0 & 0 & 0 & 0\\
0 & \frac{\zeta^{11} z_{13}^{2} z_{28} t^{169}}{z_{25}^{2} z_{33}} & \frac{\zeta^{9} z_{13}^{2} z_{31} z_{35} t^{121}}{z_{23} z_{25} z_{33} z_{34}} & 0 & 0\\
0 & 0 & 0 & 0 & 0
\end{pmatrix}
\ \ 
\begin{pmatrix}
\frac{\zeta^{9} z_{25} z_{26} z_{31} z_{35}}{t^{269}} & 0 & 0 & 0 & 0\\
\frac{\zeta^{6} z_{25} z_{26} z_{31} z_{33} z_{35}}{t^{104}} & 0 & 0 & 0 & \frac{\zeta^{3} z_{18} z_{26} z_{32} t^{148}}{z_{21} z_{23} z_{33} z_{35}^{2}}\\
0 & 0 & 0 & 0 & 0\\
0 & 0 & 0 & 0 & \frac{\zeta^{2} z_{23} z_{34} t^{91}}{z_{24} z_{35}^{2}}\\
0 & 0 & 0 & 0 & 0
\end{pmatrix}
\]
\[
\begin{pmatrix}
\frac{\zeta^{9} z_{7} z_{11} z_{16} z_{30} z_{34}}{z_{6} z_{27} t^{269}} & 0 & 0 & \frac{\zeta^{11} z_{6} z_{10} z_{21} z_{23} z_{25} z_{26} z_{30}}{z_{9} z_{18} z_{32} t^{65}} & 0\\
0 & \frac{\zeta^{2} z_{8} z_{10} z_{16}^{2} z_{26} z_{28} z_{33} z_{34} t^{226}}{z_{35}} & 0 & 0 & 0\\
0 & 0 & 0 & 0 & 0\\
\frac{\zeta^{8} z_{8} z_{10} z_{18} z_{32} z_{34}}{z_{16} z_{21} z_{23} z_{25} z_{31} z_{33} t^{161}} & 0 & 0 & 0 & 0\\
\frac{z_{8} z_{10} z_{18} z_{32}}{z_{16} z_{21} z_{23} z_{25} z_{31} t^{146}} & 0 & \frac{z_{7} z_{11} z_{16} z_{28} t^{136}}{z_{6} z_{26} z_{35}} & 0 & 0
\end{pmatrix}
\]
\[
\begin{pmatrix}
\frac{\zeta^{10} z_{6} z_{8} z_{10} z_{26} z_{30} z_{34}}{z_{27} t^{269}} & 0 & 0 & \frac{\zeta^{2} z_{10} z_{16} z_{28} z_{30}}{z_{6} z_{9} z_{31} z_{35} t^{65}} & 0\\
0 & \frac{\zeta^{4} z_{7} z_{11} z_{16}^{2} z_{21} z_{23} z_{25} z_{26} z_{31} z_{33} z_{34} t^{226}}{z_{18} z_{32}} & 0 & 0 & 0\\
0 & 0 & 0 & 0 & 0\\
\frac{\zeta^{10} z_{7} z_{11} z_{34} z_{35}}{z_{16} z_{28} z_{33} t^{161}} & 0 & 0 & 0 & 0\\
\frac{\zeta^{2} z_{7} z_{11} z_{35}}{z_{16} z_{28} t^{146}} & 0 & \frac{\zeta z_{6} z_{8} z_{10} z_{28} t^{136}}{z_{35}} & 0 & 0
\end{pmatrix}
\]
\[
\begin{pmatrix}
\frac{\zeta^{8}}{z_{15}^{2} z_{18} z_{21} z_{23} t^{269}} & \frac{\zeta^{3} z_{15} z_{18} z_{28} t^{61}}{z_{23} z_{25} z_{26} z_{31} z_{32} z_{33}^{2} z_{34} z_{35}} & \frac{\zeta^{4} z_{15} z_{21} z_{24} z_{32} z_{35}^{3} t^{13}}{z_{34}} & 0 & 0\\
\frac{\zeta^{5} z_{33}}{z_{15}^{2} z_{18} z_{21} z_{23} t^{104}} & 0 & 0 & 0 & \frac{\zeta z_{26} z_{32}^{3} z_{33} t^{148}}{z_{15}^{2} z_{18} z_{21} z_{23}}\\
0 & 0 & 0 & 0 & 0\\
\frac{\zeta^{11} z_{15} z_{23} z_{25} z_{31}}{z_{18} z_{21} z_{24} z_{34} t^{161}} & 0 & 0 & 0 & 0\\
0 & 0 & 0 & 0 & 0
\end{pmatrix}
\ \ 
\begin{pmatrix}
\frac{\zeta^{7} z_{28} z_{35}^{2}}{t^{269}} & \frac{t^{61}}{z_{28} z_{33} z_{35}^{2}} & 0 & \frac{\zeta^{7}}{t^{65}} & 0\\
\frac{\zeta^{4} z_{28} z_{33} z_{35}^{2}}{t^{104}} & 0 & 0 & \zeta^{10} z_{33} t^{100} & 0\\
0 & 0 & 0 & 0 & 0\\
0 & \frac{\zeta t^{169}}{z_{28} z_{33} z_{35}^{2}} & 0 & \zeta^{2} t^{43} & 0\\
0 & 0 & 0 & \frac{z_{33} t^{58}}{z_{34}} & 0
\end{pmatrix}
\]
\end{tiny}

We now give an overview of the method used to obtain the expressions
for $\tdet_3$ and $T_{skewcw,4}^{\boxtimes 2}$:
 
Fix bases $a_i\in A$, $b_i\in B$, and $c_i\in C$. A tensor $T\in A\ot B\ot C$
has an expression $T=\sum_{i,j,k} T^{ijk} a_i\ot b_j \ot c_k$ and is
\emph{standard tight} in this basis if there exist injective  functions
$\omega_A : [m] \to \ZZ$, $\omega_B : [m] \to \ZZ$, $\omega_C :
[m] \to \ZZ$ so that $T^{ijk}\ne 0$ implies 
$\omega_A(i)+\omega_B(j)+\omega_C(k) = 0$. In this case, we will call a choice of
$(\omega_A,\omega_B,\omega_C)$ satisfying the constraints a set of \emph{tight
weights}. Given a set of tight weights for $T$, we   consider  border rank decompositions
of the form:
\begin{equation}\label{tightwt:eqs}
  T = {\textstyle \sum}_{s=1}^r \cA_s(t) \ot \cB_s(t) \ot \cC_s(t) + O(t),
\end{equation}
where $\cA_s(t) = \sum_{i=1}^{m} \cA_{si} t^{\omega_A(i)} a_i$,
$\cB_s(t) = \sum_{j=1}^{m} \cB_{sj} t^{\omega_B(j)} b_j$, and
$\cC_s(t) = \sum_{k=1}^{m} \cC_{sk} t^{\omega_C(k)} c_k$.
Note that when the tight weights are trivial, this is an ordinary rank
decomposition. In our situation, the equations correspond
to a strict subset of the equations describing a rank decomposition, namely
those corresponding to triples $(i,j,k)$ where
$\omega_A(i)+\omega_B(j)+\omega_C(k) \le 0$. 
In the case of $T_{skewcw,4}^{\boxtimes 2}$ this reduces the number of equations down from $\binom{25+2}3=2925$
to $692$   and just as with a rank decomposition, there are
$3rm=3150$ unknowns.


We 
pick a   choice of tight weights which minimizes the number of equations to be
solved. The problem of obtaining a border rank decomposition is then split into
two questions: first, to compute a set of tight weights
$(\omega_A,\omega_B,\omega_C)$ so that $\#\{(i,j,k) \mid
\omega_A(i)+\omega_B(j)+\omega_C(k) \le 0\}$ is minimal, and second, to solve the
resulting equations \eqref{tightwt:eqs} in the $\cA_{si}$, $\cB_{sj}$, $\cC_{sk}$.

Consider the first question. Given sets $S_{\le},S_{>} \subset [m]\times
[m] \times [m]$, consider the problem of deciding if there are tight weights
$(\omega_A,\omega_B,\omega_C)$ satisfying the additional constraints that
$\omega_A(i)+\omega_B(j)+\omega_C(k) \le 0$ for $(i,j,k) \in S_{\le}$ and
$\omega_A(i)+\omega_B(j)+\omega_C(k) \ge 1$ for $(i,j,k) \in S_{>}$. These
conditions along with the original equality conditions form a linear program on the
images of $(\omega_A,\omega_B,\omega_C)$ which may be efficiently solved. There
is no harm in letting the linear program be defined over the rationals, as we may
clear denominators to obtain a solution in integers. 
One can use this fact to prune an exhaustive search of choices of $S_{\le},
S_{>}$ to find one for which $S_{\le}\cup S_{<} =  [m]\times [m]
\times [m]$, there exists a corresponding set of tight weights,
and $\# S_{\le}$ is minimal. While this is an exponential procedure, this
optimization was sufficient to solve the problem for this decomposition.

The second problem, solving the associated system, is solved with the Levenberg-Marquardt
nonlinear least squares algorithm \cite{MR10666,MR153071}.   The sparse structure of the answer
is obtained by speculatively zeroing
(or setting to simple values) coefficients   until all freedom with respect to the
equations is lost. In other words, we impose additional simple equations on the
solution and solve again until we obtain an isolated point, which can be verified by
checking that the Jacobian has full rank numerically. This procedure is repeated
many times in order to find a simple solution.  Ideally, we would
prove the resulting parameters indeed approximate an exact solution to the
equations by searching for additional relations between the parameters and then
using such relations to make symbolic methods tractable. In this case, all such
attempts failed. See \cite{CGLVkron} for further discussion of these techniques.

The border rank decomposition in this section is also a Waring border rank 
decomposition, that is, $A=B=C$, and $\cA_s(t) = \cB_s(t) = \cC_s(t)$; in particular,
$\omega_A = \omega_B = \omega_C$. This condition was imposed to make the
nonlinear search more tractable, and it also has independent interest.  The
techniques presented are equally applicable in the symmetric case as well as the
asymmetric.  

We remark that numerous relaxations of this method are possible. It was inspired by the
improved expression for $\tdet_3$, which had the structure we assume. It remains to determine
how useful it will be for more general types of tensors.

 \bibliographystyle{amsplain}

\bibliography{Lmatrix}

\def\cdprime{$''$} \def\cprime{$'$} \def\cprime{$'$} \def\cprime{$'$}
  \def\Dbar{\leavevmode\lower.6ex\hbox to 0pt{\hskip-.23ex \accent"16\hss}D}
  \def\cprime{$'$} \def\cprime{$'$} \def\cdprime{$''$} \def\cprime{$'$}
  \def\cprime{$'$} \def\Dbar{\leavevmode\lower.6ex\hbox to 0pt{\hskip-.23ex
  \accent"16\hss}D} \def\cprime{$'$} \def\cprime{$'$} \def\cprime{$'$}
  \def\cprime{$'$} \def\Dbar{\leavevmode\lower.6ex\hbox to 0pt{\hskip-.23ex
  \accent"16\hss}D} \def\cprime{$'$} \def\cprime{$'$}
\providecommand{\bysame}{\leavevmode\hbox to3em{\hrulefill}\thinspace}
\providecommand{\MR}{\relax\ifhmode\unskip\space\fi MR }
\providecommand{\MRhref}[2]{%
  \href{http://www.ams.org/mathscinet-getitem?mr=#1}{#2}
}
\providecommand{\href}[2]{#2}
\begin{thebibliography}{10}

\bibitem{2018arXiv181008671A}
J.~{Alman} and V.~{Vassilevska Williams}, \emph{{Limits on All Known (and Some
  Unknown) Approaches to Matrix Multiplication}}, 2018 IEEE 59th Annual
  Symposium on Foundations of Computer Science (2018).

\bibitem{MR3984617}
Josh Alman, \emph{Limits on the universal method for matrix multiplication},
  34th {C}omputational {C}omplexity {C}onference, LIPIcs. Leibniz Int. Proc.
  Inform., vol. 137, Schloss Dagstuhl. Leibniz-Zent. Inform., Wadern, 2019,
  pp.~Art. No. 12, 24. \MR{3984617}

\bibitem{DBLP:conf/innovations/AlmanW18}
Josh Alman and Virginia~Vassilevska Williams, \emph{Further limitations of the
  known approaches for matrix multiplication}, 9th Innovations in Theoretical
  Computer Science Conference, {ITCS} 2018, January 11-14, 2018, Cambridge, MA,
  {USA}, 2018, pp.~25:1--25:15.

\bibitem{MR3388238}
Andris Ambainis, Yuval Filmus, and Fran{\c{c}}ois Le~Gall, \emph{Fast matrix
  multiplication: limitations of the {C}oppersmith-{W}inograd method (extended
  abstract)}, S{TOC}'15---{P}roceedings of the 2015 {ACM} {S}ymposium on
  {T}heory of {C}omputing, ACM, New York, 2015, pp.~585--593. \MR{3388238}

\bibitem{MR592760}
Dario Bini, Grazia Lotti, and Francesco Romani, \emph{Approximate solutions for
  the bilinear form computational problem}, SIAM J. Comput. \textbf{9} (1980),
  no.~4, 692--697. \MR{MR592760 (82a:68065)}

\bibitem{blaserbook}
Markus Bl{\"a}ser, \emph{Fast matrix multiplication}, Graduate Surveys, no.~5,
  Theory of Computing Library, 2013.

\bibitem{MR3578455}
Markus Bl\"aser and Vladimir Lysikov, \emph{On degeneration of tensors and
  algebras}, 41st {I}nternational {S}ymposium on {M}athematical {F}oundations
  of {C}omputer {S}cience, LIPIcs. Leibniz Int. Proc. Inform., vol.~58, Schloss
  Dagstuhl. Leibniz-Zent. Inform., Wadern, 2016, pp.~Art. No. 19, 11.
  \MR{3578455}

\bibitem{MR3121848}
Weronika Buczy\'{n}ska and Jaros\l~aw Buczy\'{n}ski, \emph{Secant varieties to
  high degree {V}eronese reembeddings, catalecticant matrices and smoothable
  {G}orenstein schemes}, J. Algebraic Geom. \textbf{23} (2014), no.~1, 63--90.
  \MR{3121848}

\bibitem{BBapolar}
Weronika Buczy\'{n}ska and Jaros{\l}aw Buczy\'{n}ski, \emph{Apolarity, border
  rank and multigraded {H}ilbert scheme}, arXiv:1910.01944.

\bibitem{MR3239293}
Jaroslaw Buczy{\'n}ski and J.~M. Landsberg, \emph{On the third secant variety},
  J. Algebraic Combin. \textbf{40} (2014), no.~2, 475--502. \MR{3239293}

\bibitem{BCS}
Peter B{\"u}rgisser, Michael Clausen, and M.~Amin Shokrollahi, \emph{Algebraic
  complexity theory}, Grundlehren der Mathematischen Wissenschaften
  [Fundamental Principles of Mathematical Sciences], vol. 315, Springer-Verlag,
  Berlin, 1997, With the collaboration of Thomas Lickteig. \MR{99c:68002}

\bibitem{MR3941923}
Matthias Christandl, Fulvio Gesmundo, and Asger Kj\ae~rulff Jensen,
  \emph{Border rank is not multiplicative under the tensor product}, SIAM J.
  Appl. Algebra Geom. \textbf{3} (2019), no.~2, 231--255. \MR{3941923}

\bibitem{MR3984631}
Matthias Christandl, P\'{e}ter Vrana, and Jeroen Zuiddam, \emph{Barriers for
  fast matrix multiplication from irreversibility}, 34th {C}omputational
  {C}omplexity {C}onference, LIPIcs. Leibniz Int. Proc. Inform., vol. 137,
  Schloss Dagstuhl. Leibniz-Zent. Inform., Wadern, 2019, pp.~Art. No. 26, 17.
  \MR{3984631}

\bibitem{MR2427466}
Ciro Ciliberto and Filip Cools, \emph{On {G}rassmann secant extremal
  varieties}, Adv. Geom. \textbf{8} (2008), no.~3, 377--386. \MR{2427466}

\bibitem{CGLVkron}
Austin Conner, Fulvio Gesmundo, J.M. Landsberg, and Emanuele Ventura,
  \emph{Kronecker powers of tensors and {S}trassen's laser method},
  arXiv:1909.04785.

\bibitem{2019arXiv190909518C}
Austin {Conner}, Fulvio {Gesmundo}, Joseph~M. {Landsberg}, and Emanuele
  {Ventura}, \emph{{Tensors with maximal symmetries}}, arXiv e-prints (2019),
  arXiv:1909.09518.

\bibitem{CHLapolar}
Austin Conner, Alicia Harper, and J.M. Landsberg, \emph{Border apolarity of
  tensors {I}: New lower bounds for matrix mulitplication and $\tdet_3$},
  arXiv:1911.07981.

\bibitem{MR664715}
D.~Coppersmith and S.~Winograd, \emph{On the asymptotic complexity of matrix
  multiplication}, SIAM J. Comput. \textbf{11} (1982), no.~3, 472--492.
  \MR{664715}

\bibitem{MR91i:68058}
Don Coppersmith and Shmuel Winograd, \emph{Matrix multiplication via arithmetic
  progressions}, J. Symbolic Comput. \textbf{9} (1990), no.~3, 251--280.
  \MR{91i:68058}

\bibitem{MR3494510}
Harm Derksen, \emph{On the nuclear norm and the singular value decomposition of
  tensors}, Found. Comput. Math. \textbf{16} (2016), no.~3, 779--811.
  \MR{3494510}

\bibitem{gazka2016multigraded}
Maciej Ga{\l}azka, \emph{Apolarity, border rank and multigraded {H}ilbert
  scheme}, arXiv1601.06211.

\bibitem{MR3492642}
Nathan Ilten and Zach Teitler, \emph{Product ranks of the {$3\times3$}
  determinant and permanent}, Canad. Math. Bull. \textbf{59} (2016), no.~2,
  311--319. \MR{3492642}

\bibitem{MR3586335}
Thomas~A. Ivey and Joseph~M. Landsberg, \emph{Cartan for beginners}, Graduate
  Studies in Mathematics, vol. 175, American Mathematical Society, Providence,
  RI, 2016, Differential geometry via moving frames and exterior differential
  systems, Second edition [of MR2003610]. \MR{3586335}

\bibitem{MR2188132}
J.~M. Landsberg, \emph{The border rank of the multiplication of {$2\times2$}
  matrices is seven}, J. Amer. Math. Soc. \textbf{19} (2006), no.~2, 447--459.
  \MR{2188132 (2006j:68034)}

\bibitem{MR3729273}
\bysame, \emph{Geometry and complexity theory}, Cambridge Studies in Advanced
  Mathematics, vol. 169, Cambridge University Press, Cambridge, 2017.
  \MR{3729273}

\bibitem{MR3682743}
J.~M. Landsberg and Mateusz Micha{\l}ek, \emph{Abelian tensors}, J. Math. Pures
  Appl. (9) \textbf{108} (2017), no.~3, 333--371. \MR{3682743}

\bibitem{MR3081636}
J.~M. Landsberg and Giorgio Ottaviani, \emph{Equations for secant varieties of
  {V}eronese and other varieties}, Ann. Mat. Pura Appl. (4) \textbf{192}
  (2013), no.~4, 569--606. \MR{3081636}

\bibitem{MR3842382}
Joseph~M. Landsberg and Mateusz Micha{\l}ek, \emph{A {$2n^2-\log_2(n)-1$} lower
  bound for the border rank of matrix multiplication}, Int. Math. Res. Not.
  IMRN (2018), no.~15, 4722--4733. \MR{3842382}

\bibitem{MR3376667}
Joseph~M. Landsberg and Giorgio Ottaviani, \emph{New lower bounds for the
  border rank of matrix multiplication}, Theory Comput. \textbf{11} (2015),
  285--298. \MR{3376667}

\bibitem{LeGall:2014:PTF:2608628.2608664}
Fran\c{c}ois Le~Gall, \emph{Powers of tensors and fast matrix multiplication},
  Proceedings of the 39th International Symposium on Symbolic and Algebraic
  Computation (New York, NY, USA), ISSAC '14, ACM, 2014, pp.~296--303.

\bibitem{MR10666}
Kenneth Levenberg, \emph{A method for the solution of certain non-linear
  problems in least squares}, Quart. Appl. Math. \textbf{2} (1944), 164--168.
  \MR{10666}

\bibitem{MR153071}
Donald~W. Marquardt, \emph{An algorithm for least-squares estimation of
  nonlinear parameters}, J. Soc. Indust. Appl. Math. \textbf{11} (1963),
  431--441. \MR{153071}

\bibitem{MR765701}
Victor Pan, \emph{How to multiply matrices faster}, Lecture Notes in Computer
  Science, vol. 179, Springer-Verlag, Berlin, 1984. \MR{MR765701 (86g:65006)}

\bibitem{MR623057}
A.~Sch{\"o}nhage, \emph{Partial and total matrix multiplication}, SIAM J.
  Comput. \textbf{10} (1981), no.~3, 434--455. \MR{MR623057 (82h:68070)}

\bibitem{shitovperm3}
Y.~{Shitov}, \emph{{The Waring Rank of the 3 X 3 Permanent}},
  https://vixra.org/abs/2007.0061.

\bibitem{MR3146566}
A.~V. Smirnov, \emph{The bilinear complexity and practical algorithms for
  matrix multiplication}, Comput. Math. Math. Phys. \textbf{53} (2013), no.~12,
  1781--1795. \MR{3146566}

\bibitem{2014arXiv1412.1687S}
A.~V. {Smirnov}, \emph{{The Approximate Bilinear Algorithm of Length 46 for
  Multiplication of 4 x 4 Matrices}}, ArXiv e-prints (2014).

\bibitem{stothers}
A.~Stothers, \emph{On the complexity of matrix multiplication}, PhD thesis,
  University of Edinburgh, 2010.

\bibitem{Strassen505}
V.~Strassen, \emph{Rank and optimal computation of generic tensors}, Linear
  Algebra Appl. \textbf{52/53} (1983), 645--685. \MR{85b:15039}

\bibitem{MR882307}
\bysame, \emph{Relative bilinear complexity and matrix multiplication}, J.
  Reine Angew. Math. \textbf{375/376} (1987), 406--443. \MR{MR882307
  (88h:11026)}

\bibitem{MR1341854}
\bysame, \emph{Algebra and complexity}, First {E}uropean {C}ongress of
  {M}athematics, {V}ol.\ {II} ({P}aris, 1992), Progr. Math., vol. 120,
  Birkh\"auser, Basel, 1994, pp.~429--446. \MR{1341854}

\bibitem{Strassen493}
Volker Strassen, \emph{Gaussian elimination is not optimal}, Numer. Math.
  \textbf{13} (1969), 354--356. \MR{40 \#2223}

\bibitem{MR2652318}
Toshio Sumi, Mitsuhiro Miyazaki, and Toshio Sakata, \emph{About the maximal
  rank of 3-tensors over the real and the complex number field}, Ann. Inst.
  Statist. Math. \textbf{62} (2010), no.~4, 807--822. \MR{2652318}

\bibitem{Williams}
Virginia Williams, \emph{Breaking the coppersimith-winograd barrier}, preprint.

\end{thebibliography}

\end{document}